\documentclass{amsart}

\usepackage[T1]{fontenc}
\usepackage{enumerate, amsmath, amsfonts, amssymb, amsthm, mathrsfs, wasysym, graphics, graphicx, xcolor, url, hyperref, hypcap,  shuffle, xargs, multicol, overpic, pdflscape, multirow, hvfloat, minibox, accents, array, multido, xifthen, a4wide, ae, aecompl, blkarray, pifont, mathtools}
\usepackage{marginnote}
\hypersetup{colorlinks=true, citecolor=darkblue, linkcolor=darkblue}
\usepackage[all]{xy}
\usepackage[bottom]{footmisc}
\usepackage{tikz}
\usepackage{tkz-graph}
\usetikzlibrary{trees, decorations, decorations.markings, shapes, arrows, matrix, calc, fit, intersections, patterns, angles}
\graphicspath{{figures/}}
\makeatletter\def\input@path{{figures/}}\makeatother
\usepackage{caption}
\newtheorem{theorem}{Theorem}%[section]
\newtheorem{corollary}[theorem]{Corollary}
\newtheorem{proposition}[theorem]{Proposition}
\newtheorem{lemma}[theorem]{Lemma}

\newtheorem*{theorem*}{Theorem}%[section]

\theoremstyle{definition}

\newtheorem{remark}[theorem]{Remark}

\newcommand{\R}{\mathbb{R}} % reals
\newcommand{\fS}{\mathfrak{S}} % symmetric group
\renewcommand{\b}[1]{\mathbf{#1}} % bold letters
\newcommand{\set}[2]{\left\{ #1 \;\middle|\; #2 \right\}} % set notation
\newcommand{\bigset}[2]{\big\{ #1 \;\big|\; #2 \big\}} % big set notation
\newcommand{\ssm}{\smallsetminus} % small set minus
\newcommand{\dotprod}[2]{\left\langle \, #1 \; \middle| \; #2 \, \right\rangle} % dot product
\newcommand{\one}{{1\!\!1}} % the all one vector
\newcommand{\eqdef}{\mbox{\,\raisebox{0.2ex}{\scriptsize\ensuremath{\mathrm:}}\ensuremath{=}\,}} % :=
\newcommand{\defeq}{\mbox{~\ensuremath{=}\raisebox{0.2ex}{\scriptsize\ensuremath{\mathrm:}} }} % =:
\DeclareMathOperator{\conv}{conv} % convex hull
\DeclareMathOperator{\inv}{inv} % inversions
\newcommand{\fref}[1]{Figure~\ref{#1}} % reference figures
\newcommand{\ie}{\textit{i.e.}~} % id est
\definecolor{darkblue}{rgb}{0,0,0.7} % darkblue color
\definecolor{green}{RGB}{57,181,74} % darkblue color
\definecolor{violet}{RGB}{147,39,143} % darkblue color
\newcommand{\darkblue}{\color{darkblue}} % darkblue command
\newcommand{\defn}[1]{\textsl{\darkblue #1}} % emphasis of a definition
\newcommand{\identity}{12 \dots n} % identity

\usepackage{todonotes}

\newcommand{\meet}{\wedge} % meet
\newcommand{\join}{\vee} % join
\newcommand{\Pos}{\mathrm{Pos}} % poset of regions

\newcommand{\shard}{\Sigma}
\newcommand{\shards}{\boldsymbol{\Sigma}}
\newcommandx{\ray}[1][1=r]{\mathbf{#1}} % ray
\newcommandx{\rays}[1][1=R]{\mathbf{#1}} % rays
\newcommand{\hyp}{\mathbf{H}} % hyperplane
\newcommand{\Hyp}[1]{\mathbf{H}_{#1}} % hyperplane
\newcommand{\HA}{\mathcal{H}} % hyperplane arrangement
\newcommandx{\Perm}[1][1=n]{\mathsf{Perm}(#1)} % permutahedron
\newcommandx{\Asso}[2][1=n]{\mathsf{Asso}(#1)} % associahedron
\newcommandx{\Zono}[2][1=n]{\mathsf{Zono}(#1)} % zonotope
\newcommandx{\Fan}[1][1=F]{\mathcal{#1}} % fan

\title{Quotientopes}

\thanks{VP was partially supported by the French ANR (grants SC3A~15\,CE40\,0004\,01 and CAPPS~17\,CE40\,0018). \\ \hspace*{.5cm} FS was partially supported by the Spanish Ministry of Science (grants MTM2014-54207-P and MTM2017-83750-P), by the Einstein Foundation Berlin (grant EVF-2015-230), and by the NSF (grant DMS-1440140) while he was in residence at MSRI Berkeley in the fall of 2017.}

\author{Vincent Pilaud}
\address[VP]{CNRS \& LIX, \'Ecole Polytechnique, Palaiseau}
\email{vincent.pilaud@lix.polytechnique.fr}
\urladdr{\url{http://www.lix.polytechnique.fr/~pilaud/}}

\author{Francisco Santos}
\address[FS]{Universidad de Cantabria}
\email{francisco.santos@unican.es}
\urladdr{\url{http://personales.unican.es/santosf}}

\begin{document}

\begin{abstract}
For any lattice congruence of the weak order on~$\fS_n$, N.~Reading proved that glueing together the cones of the braid fan that belong to the same congruence class defines a complete fan. We prove that this fan is the normal fan of a polytope.

\medskip
\noindent
\textsc{msc classes.} 52B11, 52B12, 03G10, 06B10
\end{abstract}

\vspace*{-.4cm}

\maketitle

\vspace{-.4cm}

%%%%%%%%%%%%%%%%%%%%%%%%%%%%%%%%%%%%%%

\section{Introduction}

Denote by~$\fS_n$ the set of permutations of~$[n] \eqdef \{1, \dots, n\}$. We consider the classical \defn{weak order} on~$\fS_n$ defined by inclusion of inversion sets. That is $\sigma \le \tau$ if and only if $\inv(\sigma) \subseteq \inv(\tau)$ where $\inv(\sigma) \eqdef \set{(\sigma(i), \sigma(j))}{1 \le i < j \le n \text{ and } \sigma(i) > \sigma(j)}$. The Hasse diagram of the weak order can be seen geometrically:
\begin{enumerate}
\item as the dual graph of the \defn{braid fan} of type~$A_{n-1}$, \ie the fan defined by the arrangement of the hyperplanes~$\Hyp{ij} \eqdef \set{\b{x} \in \R^n}{x_i = x_j}$ for all~$1 \le i < j \le n$, directed from the region $x_1 < \dots < x_n$ to the opposite one,
%\item or as the graph of the \defn{permutahedron} $\Perm \eqdef \conv \set{\big( \sigma(1), \dots, \sigma(n) \big)}{\sigma \in \fS_n}$, oriented in the linear direction~$\alpha \eqdef (-n+1, -n+3, \dots, n-3, n-1)$.
\item or as the graph of the \defn{permutahedron} $\Perm \eqdef \conv \set{\big( \sigma^{-1}(1), \dots, \sigma^{-1}(n) \big)}{\sigma \in \fS_n}$, oriented in the linear direction~$\alpha \eqdef (-n+1, -n+3, \dots, n-3, n-1)$.
\end{enumerate}
See~\fref{fig:weakOrder4} for illustrations when~$n = 4$. 

We aim at studying similar geometric realizations for lattice quotients of the weak order~on~$\fS_n$.
Recall that a \defn{lattice congruence} of a lattice~$(L,\le,\meet,\join)$ is an equivalence relation on~$L$ that respects the meet and the join operations, \ie such that $x \equiv x'$ and~$y \equiv y'$ implies $x \meet y \, \equiv \, x' \meet y'$ and~$x \join y \, \equiv \, x' \join y'$. A lattice congruence~$\equiv$ automatically defines a \defn{lattice quotient}~$L/{\equiv}$ on the congruence classes of~$\equiv$ where the order relation is given by~$X \le Y$ iff there exists~$x \in X$ and~$y \in Y$ such that~$x \le y$. The meet~$X \meet Y$ (resp. the join~$X \join Y$) of two congruence classes~$X$ and~$Y$ is the congruence class of~$x \meet y$ (resp. of~$x \join y$) for arbitrary representatives~$x \in X$~and~$y \in Y$.

Several examples of relevant combinatorial structures arise from lattice quotients of the weak order. The fundamental example is the Tamari lattice introduced by D.~Tamari in~\cite{Tamari}. It can be defined on different Catalan families (Dyck paths, binary trees, triangulations, non-crossing partitions, etc), and its cover relations correspond to local moves in these structures (exchange, rotation, flip, etc). The Tamari lattice can also be interpreted as the quotient of the weak order by the sylvester congruence on~$\fS_n$ defined as the transitive closure of the rewriting rule~$UacVbW \equiv^\textrm{sylv} UcaVbW$ where~${a < b < c}$ are letters while~$U,V,W$ are words of~$[n]$. See \fref{fig:latticeQuotient4} for an illustration when~$n = 4$. This congruence has been widely studied in connection to geometry and algebra~\cite{Loday, LodayRonco, HivertNovelliThibon-algebraBinarySearchTrees}. Among many other examples of relevant lattice quotients of the weak order, let us mention the (type~$A$) Cambrian lattices~\cite{Reading-CambrianLattices, ChatelPilaud}, the boolean lattice, the permutree lattices~\cite{PilaudPons-permutrees}, the increasing flip lattice on acyclic twists~\cite{Pilaud-brickAlgebra}, the rotation lattice on diagonal rectangulations~\cite{LawReading, Giraudo}, etc.

In his vast study of lattice congruences of the weak order, N.~Reading observed that ``\emph{lattice congruences on the weak order know a lot of combinatorics and geometry}'' \cite[Sect.~10.7]{Reading-FiniteCoxeterGroupsChapter}. Geometrically, he showed that each lattice congruence~$\equiv$ of the weak order is realized by a complete fan~$\Fan_\equiv$ that we call \defn{quotient fan}. Its maximal cones correspond to the congruence classes of~$\equiv$ and are just obtained by glueing together the cones of the braid fan corresponding to permutations that belong to the same congruence class of~$\equiv$. Although this result was stated in a much more general context (that of lattice congruences on lattice of regions of hyperplane arrangements), we restrict our discussion to lattice quotients of the weak order on~$\fS_n$.

\begin{theorem}[\cite{Reading-HopfAlgebras}]
\label{thm:fanQuotient}
For any lattice congruence~$\equiv$ of the weak order on~$\fS_n$, the cones obtained by glueing together the cones of the braid fan that belong to the same congruence class of~$\equiv$ form a fan~$\Fan_\equiv$ whose dual graph coincides with the Hasse diagram of the quotient of the weak order by~${\equiv}$.
\end{theorem}

\begin{figure}
	\capstart
	\centerline{\includegraphics[scale=.6]{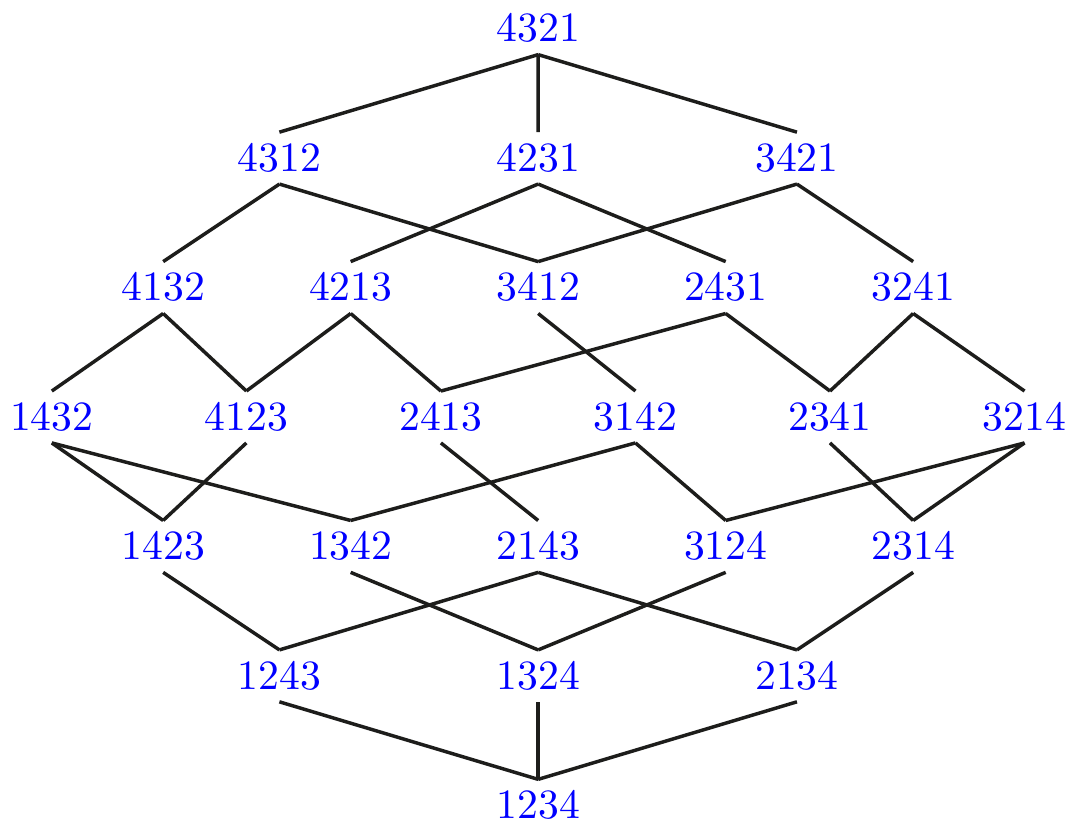} \; \includegraphics[scale=.6]{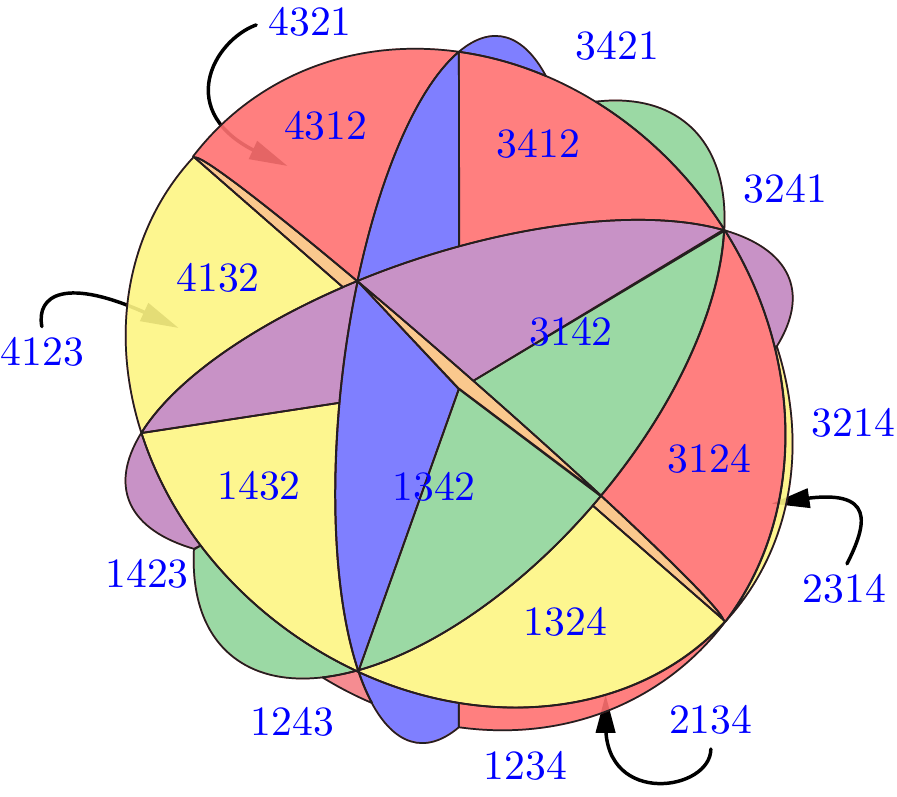} \; \includegraphics[scale=.6]{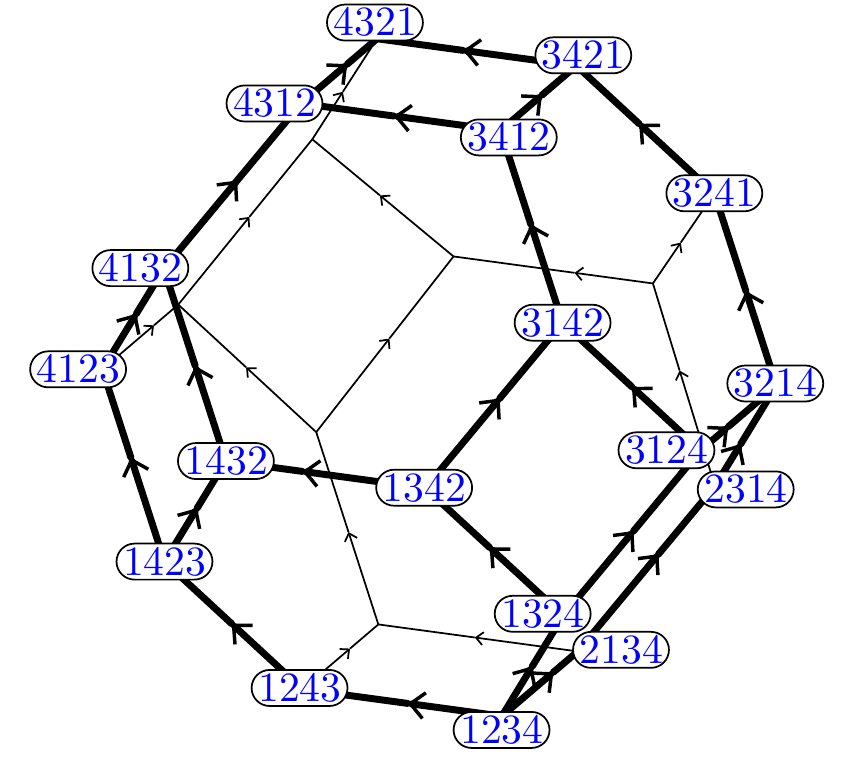}}
	\caption{The Hasse diagram of the weak order on~$\fS_4$ (left) can be seen as the dual graph of braid fan (middle) or as an orientation of the graph of the permutahedron~$\Perm[4]$ (right).}
	\vspace{-.3cm}
	\label{fig:weakOrder4}
\end{figure}

\begin{figure}
	\capstart
	\centerline{\includegraphics[scale=.6]{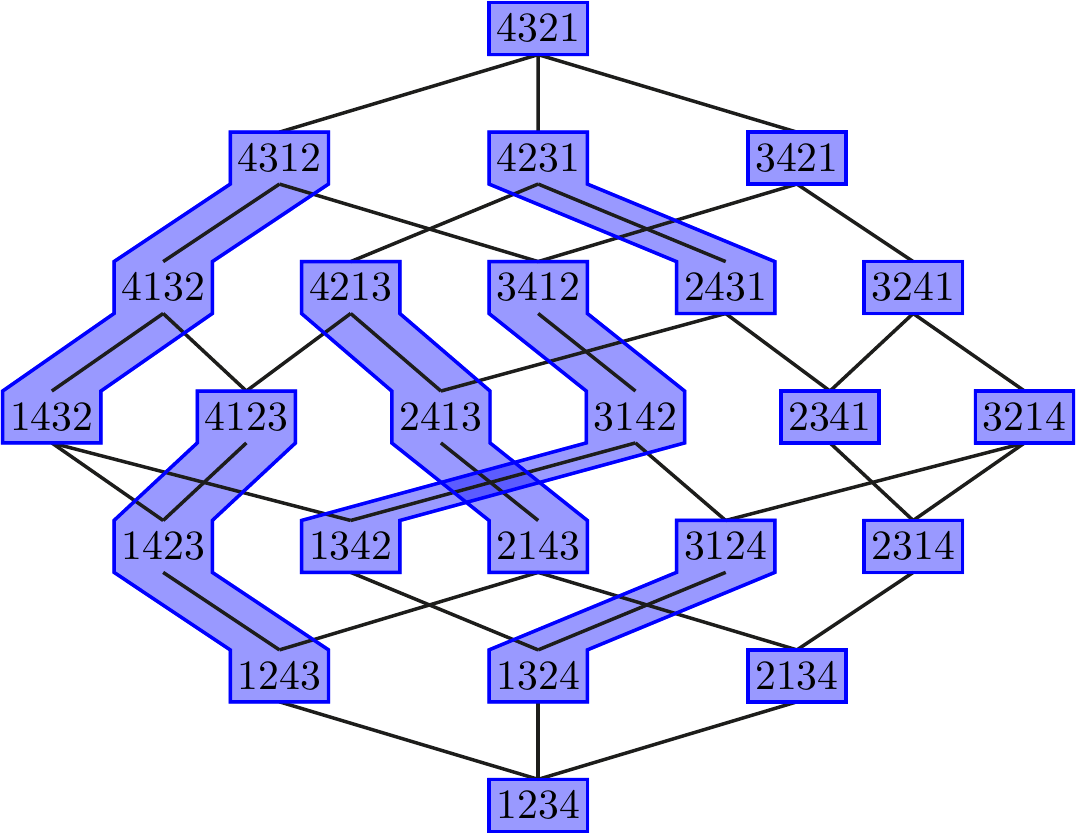} \qquad \includegraphics[scale=.48]{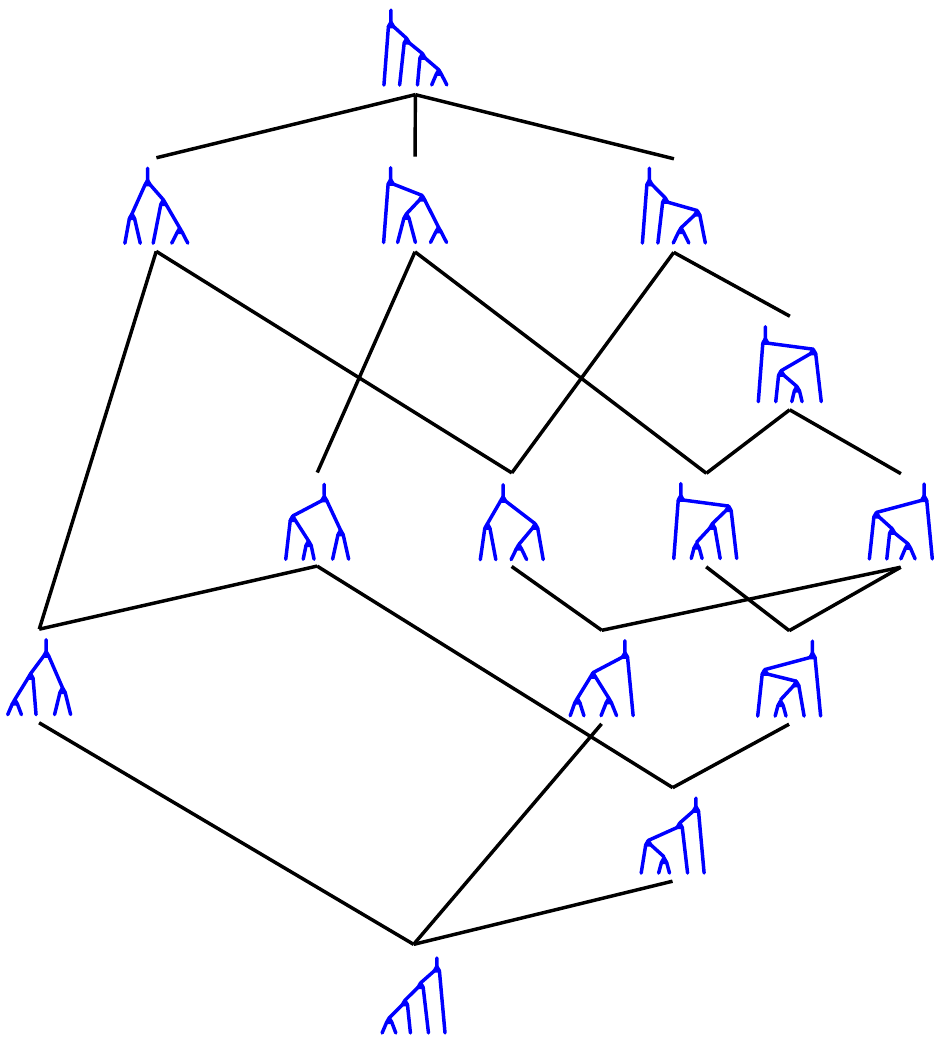}}
	\caption{The Tamari lattice (right) is the quotient of the weak order by the sylvester congruence~$\equiv^\textrm{sylv}$~(left). Each congruence class is given by a blue box on the left and corresponds to a binary tree on the right.}
	\vspace{-.3cm}
	\label{fig:latticeQuotient4}
\end{figure}

However, as observed by N.~Reading in~\cite{Reading-HopfAlgebras}, ``\emph{this theorem gives no means of knowing when~$\Fan_\equiv$ is the normal fan of a polytope}''. For the above-mentioned examples of lattice congruences, this problem was settled by specific constructions of polytopes realizing the quotient fan~$\Fan_\equiv$: J.-L.~Loday's associahedron~\cite{Loday} for the Tamari lattice, C.~Hohlweg and C.~Lange's associahedra~\cite{HohlwegLange, LangePilaud} for the Cambrian lattices, cubes for the boolean lattices, permutreehedra~\cite{PilaudPons-permutrees} for the permutree lattices, brick polytopes~\cite{PilaudSantos-brickPolytope} for increasing flip lattices on acyclic twists, Minkowski sums of opposite associahedra for rotation lattices on diagonal rectangulations~\cite{LawReading}, etc. Although these realizations have similarities, each requires an independent construction and proof. In particular, the intersection of the half-spaces defining facets of the classical permutahedron normal to the rays of~$\Fan_\equiv$ does not realize~$\Fan_\equiv$ in general, in contrast to the specific situation of~\cite{Loday, HohlwegLange, LangePilaud, PilaudPons-permutrees}. Our contribution is to provide a general method to construct a polytope~$P_\equiv$ whose normal fan is the quotient fan~$\Fan_\equiv$. We therefore prove the following statement.

\begin{theorem}
\label{thm:quotientopes}
For any lattice congruence~$\equiv$ of the weak order on~$\fS_n$, the fan~$\Fan_\equiv$ obtained by glueing the braid fan according to the congruence classes of~${\equiv}$ is the normal fan of a polytope. In particular, the graph of this polytope is the Hasse diagram of the quotient of the weak order by~${\equiv}$.
\end{theorem}

We call \defn{quotientopes} the resulting polytopes. Examples are illustrated in Figures~\ref{fig:permAssoCube}, \ref{fig:relevantQuotientopes} and~\ref{fig:quotientopeLattice}.

%%%%%%%%%%%%%%%%%%%%%%%%%%%%%%%%%%%%%%

\section{Background}

\subsection{Polyhedral geometry}

We briefly recall basic definitions and properties of polyhedral fans and polytopes, and refer to~\cite{Ziegler-polytopes} for a classical textbook on this topic.

A hyperplane~$H \subset \R^d$ is a \defn{supporting hyperplane} of a set~$X \subset \R^d$ if~$H \cap X \ne \varnothing$ and~$X$ is contained in one of the two closed half-spaces of~$\R^d$ defined by~$H$.

We denote by~$\R_{\ge0}\rays \eqdef \set{\sum_{\ray \in \rays} \lambda_{\ray} \, \ray}{\lambda_{\ray} \in \R_{\ge0}}$ the \defn{positive span} of a set~$\rays$ of vectors of~$\R^d$.
A \defn{polyhedral cone} is a subset of~$\R^d$ defined equivalently as the positive span of finitely many vectors or as the intersection of finitely many closed linear halfspaces.
The \defn{faces} of a cone~$C$ are the intersections of~$C$ with the supporting hyperplanes of~$C$.
The $1$-dimensional (resp.~codimension~$1$) faces of~$C$ are called~\defn{rays} (resp.~\defn{facets}) of~$C$.
A cone is \defn{simplicial} if it is generated by a set of independent vectors.

A \defn{polyhedral fan} is a collection~$\Fan$ of polyhedral cones such that
\begin{itemize}
\item if~$C \in \Fan$ and~$F$ is a face of~$C$, then~$F \in \Fan$,
\item the intersection of any two cones of~$\Fan$ is a face of both.
\end{itemize}
A fan is \defn{simplicial} if all its cones are, and \defn{complete} if the union of its cones covers the ambient space~$\R^d$.
For two fans~$\Fan, \Fan[G]$ in~$\R^d$, we say that~$\Fan$ \defn{refines}~$\Fan[G]$ (and that~$\Fan[G]$ \defn{coarsens}~$\Fan$) if every cone of~$\Fan$ is contained in a cone of~$\Fan[G]$.

A \defn{polytope} is a subset~$P$ of~$\R^d$ defined equivalently as the convex hull of finitely many points or as a bounded intersection of finitely many closed affine halfspaces.
The \defn{dimension}~$\dim(P)$ is the dimension of the affine hull of~$P$.
The \defn{faces} of~$P$ are the intersections of~$P$ with its supporting hyperplanes.
The dimension~$0$ (resp.~dimension~$1$, resp.~codimension~$1$) faces are called \defn{vertices} (resp.~\defn{edges}, resp.~\defn{facets}) of~$P$.
A polytope is \defn{simple} if its supporting hyperplanes are in general position, meaning that each vertex is incident to $\dim(P)$ facets (or equivalently to $\dim(P)$ edges).

The (outer) \defn{normal cone} of a face~$F$ of~$P$ is the cone generated by the outer normal vectors of the facets of~$P$ containing~$F$.
In other words, it is the cone of vectors~$\ray[c]$ such that the linear form~${\b{x} \mapsto \dotprod{\ray[c]}{\b{x}}}$ on~$P$ is maximized by all points of the face~$F$.
The (outer) \defn{normal fan} of~$P$ is the collection of the (outer) normal cones of all its faces.
We say that a complete polyhedral fan in~$\R^d$ is \defn{polytopal} when it is the normal fan of a polytope of~$\R^d$.
A classical characterization of polytopality of complete simplicial fans can be obtained as a reformulation of regularity of triangulations of vector configurations, as introduced in the theory of secondary polytopes~\cite[Chap.~7]{GelfandKapranovZelevinsky}, see also~\cite[Chap.~5]{DeLoeraRambauSantos}.
Here, we present a reformulation of this characterization to deal with (not necessarily simplicial) fans that coarsen a complete simplicial fan.

\begin{proposition}
\label{prop:polytopalSubfanFan}
Consider two fans~$\Fan, \Fan[G]$ of~$\R^d$, and let~$\rays \subset \R^d$ be a set of representative vectors for the rays of~$\Fan$. Assume that~$\Fan$ is complete and simplicial, and that~$\Fan$ refines~$\Fan[G]$. Then the following assertions are equivalent:
\begin{enumerate}
\item $\Fan[G]$ is the normal fan of a polytope in~$\R^d$.
\item There exists a map~$h: \rays \to \R_{>0}$ with the property that for any~$\ray, \ray' \in \rays$ and~$\rays[S] \subset \rays$ for which~$C \eqdef \R_{\ge0} (\rays[S] \cup \{\ray\})$ and~$C' \eqdef \R_{\ge0} (\rays[S] \cup \{\ray'\})$ are two adjacent maximal cones of~$\Fan$, if 
\[
\alpha \, \ray + \alpha' \, \ray' + \sum_{\ray[s] \in \rays[S]} \beta_{\ray[s]} \, \ray[s] = 0
\]
is the unique (up to rescaling) linear dependence with~$\alpha, \alpha' > 0$ among~${\{\ray, \ray'\} \cup \rays[S]}$ then
\[
\alpha \, h(\ray) + \alpha' \, h(\ray') + \sum_{\ray[s] \in \rays[S]} \beta_{\ray[s]} \, h(\ray[s]) \ge 0,
\]
with equality if and only if the cones~$C$ and~$C'$ are contained in the same cone of~$\Fan[G]$.
\end{enumerate}
Under these conditions, $\Fan[G]$ is the normal fan of the polytope defined by
\[
\bigset{\b{x} \in \R^d}{\dotprod{\ray}{\b{x}} \le h(\ray) \text{ for all } \ray \in \rays}.
\]
\end{proposition}

\begin{proof}
The proof is similar to that of \cite[Lem.~2.1]{ChapotonFominZelevinsky}, and we just adapt it here for the convenience of the reader.
Assume first that~$\Fan[G]$ is the normal fan of a polytope~$P \subseteq \R^d$ and define~$h : \rays \to \R_{>0}$ by~$h(\ray) \eqdef \max \set{\dotprod{\ray}{\b{x}}}{\b{x} \in P}$.
Consider~$\ray, \ray' \in \rays$ and~$\rays[S] \subset \rays$ such that the cones~$C \eqdef \R_{\ge0} (\rays[S] \cup \{\ray\})$ and~$C' \eqdef \R_{\ge0} (\rays[S] \cup \{\ray'\})$ are two adjacent maximal cones of~$\Fan$. Let~$\b{v}$ and~$\b{v}'$ be the vertices of~$P$ whose normal cones contain~$C$ and~$C'$ respectively. Then by definition,
\[
h(\ray) = \dotprod{\ray}{\b{v}},
\qquad
h(\ray') = \dotprod{\ray'}{\b{v}'}
\qquad\text{and}\qquad
h(\ray[s]) = \dotprod{\ray[s]}{\b{v}} = \dotprod{\ray[s]}{\b{v}'}
\quad\text{for all~$\ray[s] \in \rays[S]$.}
\]
Therefore, applying the linear form~$\ray[s] \mapsto \dotprod{\ray[s]}{\b{v}}$ to the linear dependance
\(
\alpha \, \ray + \alpha' \, \ray' + \sum \beta_{\ray[s]} \, \ray[s] = 0
\)
among the vectors of~$\rays[S] \cup \{\ray, \ray'\}$, we obtain
\(
\alpha \dotprod{\ray}{\b{v}} + \alpha' \dotprod{\ray'}{\b{v}} + \sum \beta_{\ray[s]} \dotprod{\ray[s]}{\b{v}} = 0.
\)
If~$C$ and~$C'$ belong to the same cone of~$\Fan[G]$, then~$\b{v} = \b{v}'$ and~$h(\ray') = \dotprod{\ray'}{\b{v}}$, thus we obtain the equality
\(
\alpha \, h(\ray) + \alpha' \, h(\ray') + \sum \beta_{\ray[s]} \, h(\ray[s]) = 0.
\)
Otherwise, we have~$h(\ray') = \dotprod{\ray'}{\b{v}'} > \dotprod{\ray'}{\b{v}}$ and since~$\alpha' > 0$ we obtain the inequality
\(
\alpha \, h(\ray) + \alpha' \, h(\ray') + \sum \beta_{\ray[s]} \, h(\ray[s]) > 0.
\)

Reciprocally, assume that there exists a height function~$h : \rays \to \R_{>0}$ such that \mbox{Condition~(2)} is satisfied and consider the polytope~$P \eqdef \set{\b{x} \in \R^d}{\dotprod{\ray}{\b{x}} \le h(\ray) \text{ for all } \ray \in \rays}$.
For each maximal cone ${C = \R_{\ge 0} \rays[S]}$ of~$\Fan[F]$, let~$\b{v}_C$ be the intersection of the hyperplanes~$\set{\b{x} \in \R^d}{\dotprod{\ray[s]}{\b{x}} = h(\ray[s])}$ for~$\ray[s] \in \rays[S]$.
Let~$\ray, \ray' \in \rays$ and~$\rays[S] \subset \rays$ for which~$C \eqdef \R_{\ge0} (\rays[S] \cup \{\ray\})$ and~${C' \eqdef \R_{\ge0} (\rays[S] \cup \{\ray'\})}$ are two adjacent maximal cones of~$\Fan$. 
Condition~(2) implies that~$\dotprod{\ray}{\b{v}_{C}} \ge \dotprod{\ray}{\b{v}_{C'}}$ with equality if and only if~$C$ and~$C'$ are contained in the same cone of~$\Fan[G]$.
Therefore, for any vector~$\ray[c] \in \R^d$ located on the side of~$\ray$ of the hyperplane spanned by~$\rays[S]$, we can write~$\ray[c]$ as a linear combination~$\ray[c] = \gamma_{\ray} \, \ray + \sum_{\ray[s] \in \rays[S]} \gamma_{\ray[s]} \, \ray[s]$ with~$\gamma_{\ray} > 0$, and we obtain that
\[
\dotprod{\ray[c]}{\b{v}_C} = \gamma_{\ray} \dotprod{\ray}{\b{v}_C} + \sum_{\ray[s] \in \rays[S]} \gamma_{\ray[s]} \dotprod{\ray[s]}{\b{v}_C} \ge \gamma_{\ray} \dotprod{\ray}{\b{v}_{C'}} + \sum_{\ray[s] \in \rays[S]} \gamma_{\ray[s]} \dotprod{\ray[s]}{\b{v}_{C'}} = \dotprod{\ray[c]}{\b{v}_{C'}},
\]
with equality if and only if~$C$ and~$C'$ are contained in the same cone of~$\Fan[G]$.
Consider now a vector~$\ray[c] \in \R^d$ and a maximal cone~$C$ of~$\Fan[F]$.
For dimension reason, there exists a line segment~$L$ joining~$\ray[c]$ with some interior point of~$C$ and not passing through any cone of codimension two or more in~$\Fan[F]$.
Let~$C_1, C_2, \dots, C_k = C$ denote the cones of~$\Fan[F]$ along~$L$ (with~$\ray[c] \in C_1$).
By the previous observation, we obtain
\[
\dotprod{\ray[c]}{\b{v}_{C_1}} \ge \dotprod{\ray[c]}{\b{v}_{C_2}} \ge \dots \ge \dotprod{\ray[c]}{\b{v}_C},
\]
with equalities if and only if~$C_i$ and~$C_{i+1}$ are contained in the same cone of~$\Fan[G]$.
Therefore, the linear form~$\b{x} \mapsto \dotprod{\ray[c]}{\b{x}}$ on~$P$ is maximized by~$\b{v}_C$ if and only if~$\ray[c]$ belongs to the cone of~$\Fan[G]$ containing the cone~$C$.
We conclude that for any maximal cone~$C$ of~$\Fan$, the point~$\b{v}_C$ is a vertex of~$P$ whose normal cone is the cone of~$\Fan[G]$ containing~$C$, which shows that~$\Fan[G]$ is the normal fan of~$P$.
\end{proof}

\subsection{Braid fan}

We consider the \defn{braid arrangement}~$\HA_n \eqdef \set{\Hyp{ij}}{1 \le i < j \le n}$ consisting of the hyperplanes of the form~$\Hyp{ij} \eqdef \set{\b{x} \in \R^n}{x_i = x_j}$.
The closures of the connected components of~$\R^n \ssm \bigcup \HA_n$ (together with all their faces) form a fan.
This fan is complete and simplicial, but not essential (all its cones contain the line~$\R \one \eqdef \R(1,1,\dots,1)$). 
We call \defn{braid fan}, and denote by~$\Fan_n$, the intersection of this fan with the hyperplane~$\hyp \eqdef \bigset{\b{x} \in \R^n}{\sum_{i \in [n]} x_i = 0}$.

For example, we have represented in Figures~\ref{fig:shards3}\,(left) and~\ref{fig:weakOrder4}\,(middle) the braid fan when~$n = 3$ and~$n = 4$ respectively.
As the $3$-dimensional fan~$\Fan_4$ is difficult to visualize in \fref{fig:weakOrder4}\,(middle), we also use another classical representation in \fref{fig:shards4}\,(left): we intersect~$\Fan_4$ with a unit sphere and we stereographically project the resulting arrangement of great circles from the pole~$4321$ to the plane.
Each circle then corresponds to a hyperplane~$x_i = x_j$ with~$i < j$, separating a disk where~$x_i < x_j$ from an unbounded region where~$x_i > x_j$.

%The fan~$\Fan_n$ has a $k$-dimensional cone for each surjection from~$[n]$ to~$[k+1]$: namely, a surjection ${\pi : [n] \to [k+1]}$ corresponds to the cone~$C(\pi) \eqdef \set{\b{x} \in \hyp}{\pi(i) \le \pi(j) \Rightarrow x_i \le x_j \text{ for all } i,j \in [n]}$.
The cones of the braid fan~$\Fan_n$ are naturally labeled by ordered partitions of~$[n]$: an ordered partition~$\pi = \pi_1 | \pi_2 | \dots | \pi_k$ of~$[n]$ into $k$ parts corresponds to the $(k-1)$-dimensional cone
\(
C(\pi) \eqdef \set{\b{x} \in \hyp}{x_u \le x_v \text{ for all $i \le j$, $u \in \pi_i$ and~$v \in \pi_j$}}.
\)
In particular, the fan $\Fan_n$ has 
\begin{itemize}
\item a maximal cone~$C(\sigma) \eqdef \! \set{\b{x} \in \hyp}{x_{\sigma(1)} \! \le \! x_{\sigma(2)} \! \le \! \dots \le x_{\sigma(n)} \! }$ for each permutation~${\sigma \in \fS_n}$, 
\item a ray~$C(R)$ for each subset~$R$ of~$[n]$ distinct from~$\varnothing$ and~$[n]$. Namely, if~$R = \{r_1, \dots, r_p\}$ and~$[n] \ssm R = \{s_1, \dots, s_{n-p}\}$ then~$C(R) \eqdef \set{\b{x} \in \hyp}{x_{r_1} = \dots = x_{r_p} \le x_{s_1} = \dots = x_{s_{n-p}}}$.
\end{itemize}
This is illustrated in Figures~\ref{fig:shards3}\,(left) and~\ref{fig:shards4}\,(left) when~$n = 3$ and~$n = 4$ respectively: chambers are labeled with blue permutations of~$[n]$ and rays are labeled with red subsets of~$[n]$.

Note that the fundamental chamber~$C(\identity)$ has rays labeled by the $n-1$ subsets of the form~$[k]$ with~${0 < k < n}$.
Any other chamber~$C(\sigma)$ is obtained from~$C(\identity)$ by permutation of coordinates and has thus rays labeled by~$\sigma([k])$ with~$0 < k < n$.
For example, the chamber~$C(312)$ of~$\Fan_3$ has rays labeled by~$\{3\}$ and~$\{3,1\}$, and the chamber~$C(3421)$ of~$\Fan_4$ has rays labeled by~$\{3\}$, $\{3,4\}$ and~$\{2,3,4\}$, see Figures~\ref{fig:shards3}\,(left) and~\ref{fig:shards4}\,(left).
Two permutations~$\sigma, \sigma'$ are said to be \defn{adjacent} when their cones~$C(\sigma)$ and~$C(\sigma')$ share a facet, or equivalently when~$\sigma$ and~$\sigma'$ differ by the exchange of two consecutive positions.

To understand the geometry of~$\Fan_n$, we need to choose convenient representative vectors in~$\hyp$ for the rays of~$\Fan_n$. 
We denote by~$\Delta \eqdef \{\alpha_1, \dots, \alpha_{n-1}\}$ the root basis (where~${\alpha_i \eqdef \b{e}_{i+1} - \b{e}_i}$) and by~$\nabla \eqdef \{\omega_1, \dots, \omega_{n-1}\}$ the fundamental weight basis (\ie the dual basis of the root basis~$\Delta$).
A subset~${\varnothing \ne R \subsetneq [n]}$ corresponds to the ray~$\ray(R)$ of~$\Fan_n$ whose $k$th coordinate in the fundamental weight basis is $\one_{k \in R} - \one_{k+1 \in R}$ (where~$\one_X = 1$ if the property~$X$ holds and~$0$ otherwise). 
For example, the rays of the fundamental chamber are the rays~$\ray([k]) = \omega_k$ for~$0 < k < n$. % while the rays of any chamber~$C(\sigma)$ are the rays
%The following immediate lemma is left to the reader.

\begin{lemma}
\label{lem:linearDependence}
Let~$\sigma, \sigma'$ be two adjacent permutations. Let~$\varnothing \ne R \subsetneq [n]$ (resp.~$\varnothing \ne R' \subsetneq [n]$) be such that~$\ray(R)$ (resp.~$\ray(R')$) is the ray of~$C(\sigma)$ not in~$C(\sigma')$ (resp.~of~$C(\sigma')$ not in~$C(\sigma)$). Then the linear dependence among the rays of the cones~$C(\sigma)$ and~$C(\sigma')$ is given by
\[
\ray(R) \, + \, \ray(R') \; = \; \ray(R \cap R') \, + \, \ray(R \cup R')
\]
where we set~$\ray(\varnothing) = \ray([n]) = 0$ by convention.
%We say that the dependences where~$R \cap R' = \varnothing$ or~$R \cup R' = [n]$ are \defn{degenerate}.
\end{lemma}

\begin{proof}
Since~$\sigma$ and~$\sigma'$ are adjacent permutations of~$\fS_n$, there is~$i \in [n-1]$ such that~${\sigma(i) = \sigma'(i+1)}$ and~$\sigma(i+1) = \sigma'(i)$ while~$\sigma(\ell) = \sigma'(\ell)$ for any~$\ell \notin \{i, i+1\}$.
%We have~$\ray(\sigma([i])) \in C(\sigma) \ssm C(\sigma')$ and~$\ray(\sigma'([i])) \in C(\sigma') \ssm C(\sigma)$ while~$\ray(\sigma([j]) \in C(\sigma) \cap C(\sigma')$ for any~$j \notin \{i, i+1\}$.
Let~$R \eqdef \sigma([i])$ and~$R' \eqdef \sigma'([i])$ and observe that~$R \cap R' = \sigma([i-1]) = \sigma'([i-1])$ and~$R \cup R' = \sigma([i+1]) = \sigma'([i+1])$.
Since~$\one_{k \in R} + \one_{k \in R'} = \one_{k \in R \cap R'} + \one_{k \in R \cup R'}$ for any~$k \in [n]$ and~$\ray(R) = \sum_{k \in [n-1]} (\one_{k \in R} - \one_{k+1 \in R}) \omega_k$, we obtain the linear dependence~$\ray(R) \, + \, \ray(R') \; = \; \ray(R \cap R') \, + \, \ray(R \cup R')$, where we use the convention~$\ray(\varnothing) = 0$ if~$R \cap R' = \varnothing$ and~$\ray([n]) = 0$ if~$R \cup R' = [n]$.
Moreover, we have~${\ray(R) \in C(\sigma) \ssm C(\sigma')}$ and~${\ray(R') \in C(\sigma') \ssm C(\sigma)}$ while~${\ray(R \cap R') \in C(\sigma) \cap C(\sigma')}$ and~${\ray(R \cup R') \in C(\sigma) \cap C(\sigma')}$.
Therefore, we have identified the unique (up to rescaling) linear dependence of the rays of the cones~$C(\sigma)$ and~$C(\sigma')$.
\end{proof}

For example, the linear dependence among the rays in the adjacent cones~$C(123)$ and~$C(213)$ of~$\Fan_3$ is~$\ray(\{1\}) + \ray(\{2\}) = \ray(\{12\})$, while the linear dependence among the rays in the adjacent cones~$C(123)$ and~$C(132)$ of~$\Fan_3$ is~$\ray(\{1,2\}) + \ray(\{1,3\}) = \ray(\{1\})$. See \fref{fig:shards3}\,(left).
The first non-degenerate linear dependencies (\ie where~$R \cap R' \ne \varnothing$ and~$R \cup R' \ne [n]$) arise in~$\Fan_4$: for instance, the linear dependence among the rays in the adjacent cones~$C(4132)$ and~$C(4312)$ of~$\Fan_4$ is~$\ray(\{3,4\}) + \ray(\{1,4\}) = \ray(\{4\}) + \ray(\{1,3,4\})$. See \fref{fig:shards4}\,(left).
%For example, the linear dependence among the rays in the adjacent cones~$C(4132)$ and~$C(4312)$ of~$\Fan_4$ is~$\ray(\{3,4\}) + \ray(\{1,4\}) = \ray(\{4\}) + \ray(\{1,3,4\})$, while the linear dependence among the rays in the adjacent cones~$C(4213)$ and~$C(4231)$ of~$\Fan_4$ is~$\ray(\{1,2,4\}) + \ray(\{2,3,4\}) = \ray(\{2,4\})$. See \fref{fig:shards4}\,(left).

\subsection{Shards}

\begin{figure}
	\capstart
	\centerline{\includegraphics[scale=.9]{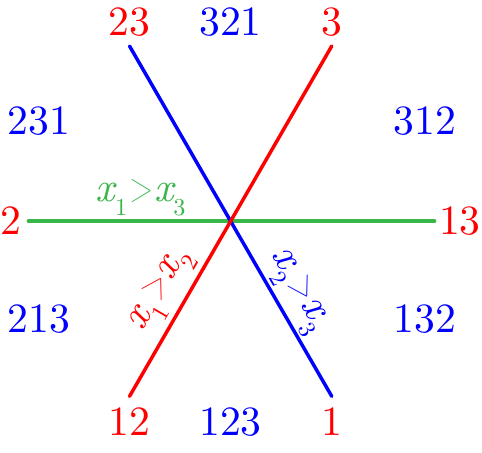} \qquad \includegraphics[scale=.9]{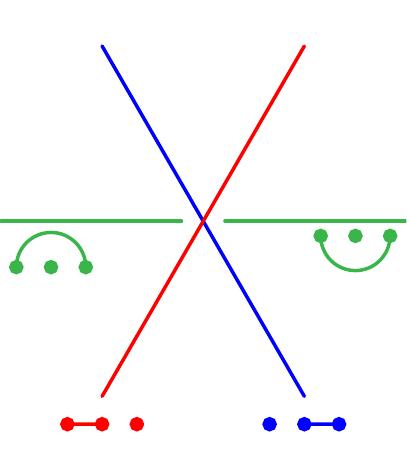} \qquad \includegraphics[scale=.9]{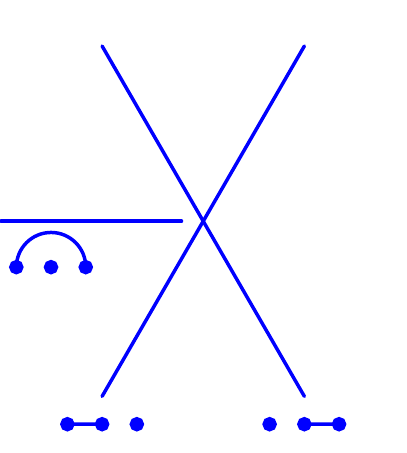}}
	\caption{The braid fan~$\Fan_3$ (left), the corresponding shards (middle), and the quotient fan given by the sylvester congruence~$\equiv^\textrm{sylv}$~(right).}
	\label{fig:shards3}
\end{figure}

\begin{figure}
	\capstart
	\centerline{\includegraphics[scale=.5]{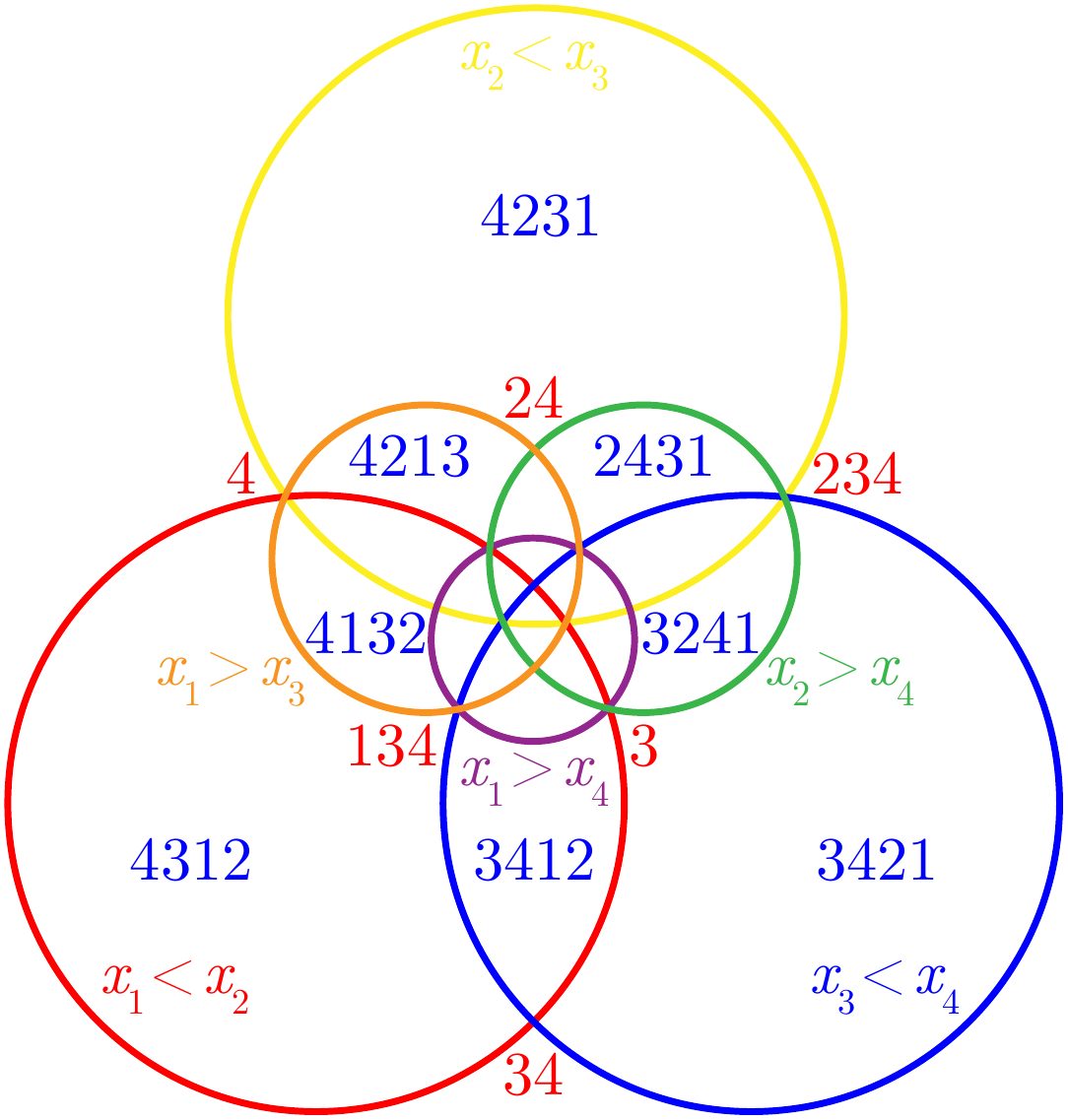} \; \includegraphics[scale=.5]{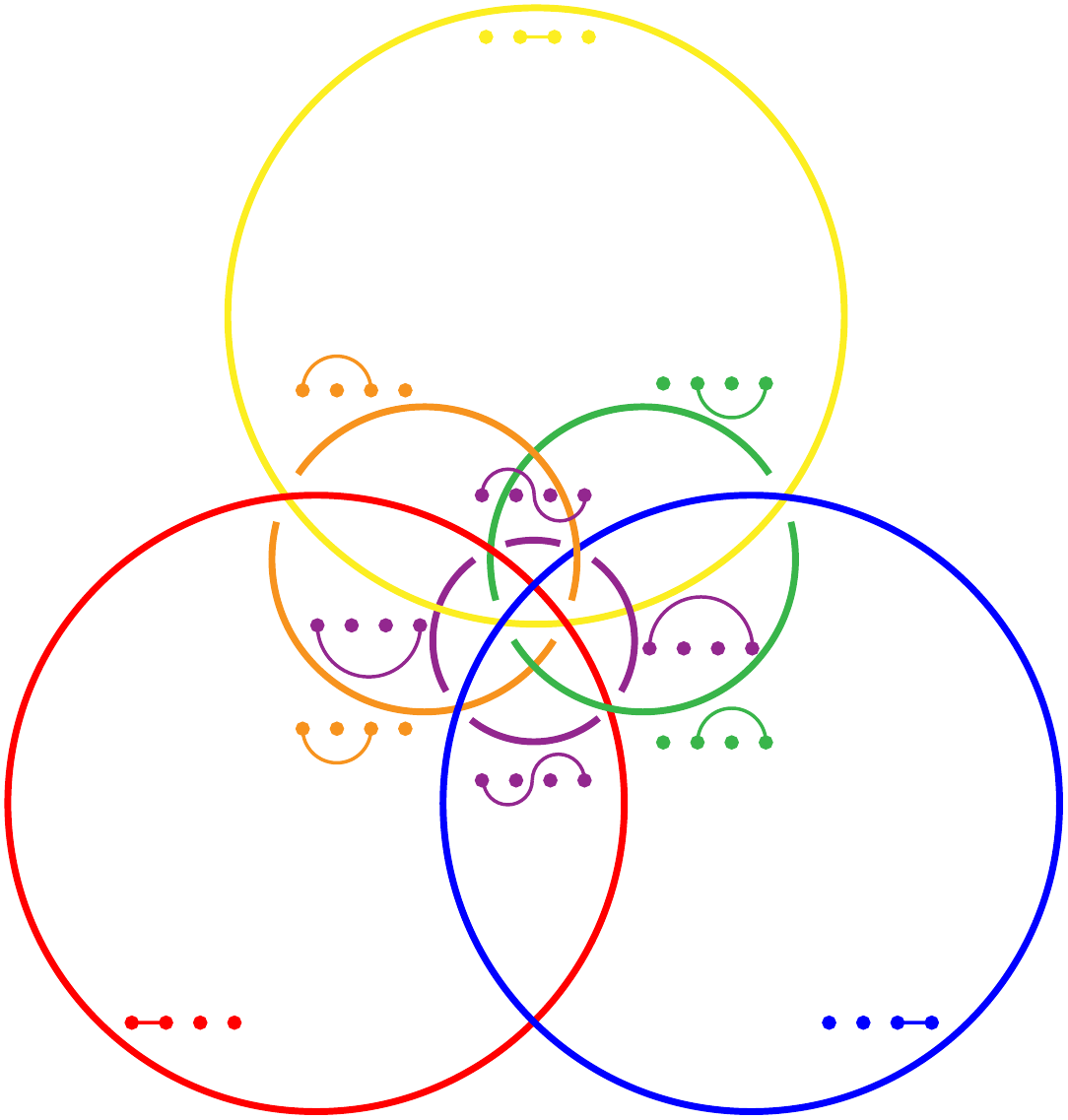} \; \includegraphics[scale=.5]{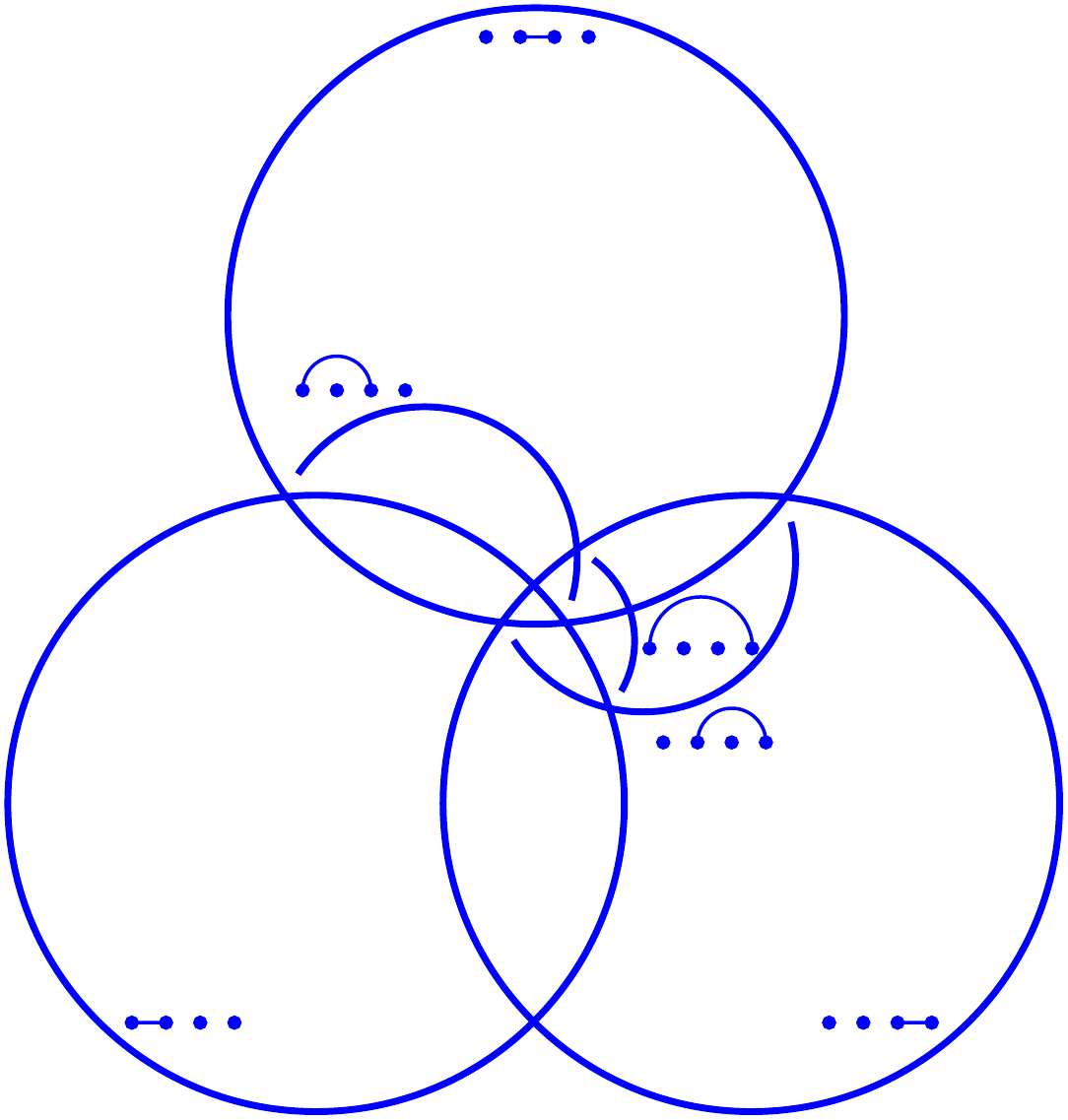}}
	\caption{A stereographic projection of the braid fan~$\Fan_4$ (left) from the pole~$4321$, the corresponding shards (middle), and the quotient fan given by the sylvester congruence~$\equiv^\textrm{sylv}$~(right).}
	\label{fig:shards4}
\end{figure}

We now briefly present shards, a powerful tool to deal with lattice quotients of the weak order with a geometric perspective. Shards were introduced by N.~Reading~\cite{Reading-posetRegions}, see also his recent survey chapters~\cite{Reading-PosetRegionsChapter, Reading-FiniteCoxeterGroupsChapter}. For any~$1 \le i < j \le n$, let~$[i,j] \eqdef \{i, \dots, j\}$ and~${]i,j[} \eqdef \{i+1, \dots, j-1\}$. For any~$S \subseteq {]i,j[}$, the \defn{shard}~$\shard(i,j,S)$ is the cone
\[
\shard(i,j,S) \eqdef \set{\b{x} \in \R^n}{x_i = x_j, \; x_i \ge x_k \text{ for all } k \in S, \; x_i \le x_k \text{ for all } k \in {]i,j[} \ssm S}.
\]
The hyperplane~$\Hyp{ij}$ is decomposed into the~$2^{j-i-1}$~shards~$\shard(i,j,S)$ for all subsets~${S \subseteq {]i,j[}}$. The shards thus have to be thought of as pieces of the hyperplanes of the braid arrangement.
Let
\[
\shards_n \eqdef \set{\shard(i,j,S)}{1 \le i < j \le n \text{ and } S \subseteq {]i,j[}}
\]
denote the collection of all shards of the braid arrangement in~$\R^n$.

We have illustrated the shard decomposition in Figures~\ref{fig:shards3}\,(middle) and~\ref{fig:shards4}\,(middle).
In what follows, we use a convenient notation borrowed from N.~Reading's work on arc diagrams~\cite{Reading-arcDiagrams}: the shard~$\shard(i,j,S)$ is labeled by an arc joining the $i$th dot to the $j$th dot and passing above (resp.~below) the $k$th dot when~$k \in S$ (resp.~when~$k \notin S$).
For instance, the arc~\raisebox{-.16cm}{\includegraphics[scale=.8]{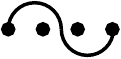}} represents the shard~$\shard(1,4,\{2\})$.

Before going further, we state two technical lemmas connecting rays and shards.

\begin{lemma}
\label{lem:rayInShard}
For any~$\varnothing \ne R \subsetneq [n]$, any~$1 \le i < j \le n$ and any~$S \subseteq {]i,j[}$, the ray~$\ray(R)$ lies in the shard~$\shard(i,j,S)$ if and only if either~$\{i,j\} \subseteq R$ and~$S \subseteq {]i,j[} \cap R$, or~$\{i,j\} \subseteq [n] \ssm R$ and~${]i,j[} \cap R \subseteq S$.
\end{lemma}

\begin{proof}
Recall that
\begin{itemize}
\item the ray~$\ray(R)$ lies on the open half-line~$\set{\b{x} \in \hyp}{x_{r_1} = \dots = x_{r_p} < x_{s_1} = \dots = x_{s_{n-p}}}$ where $R = \{r_1, \dots, r_p\}$ and~$[n] \ssm R = \{s_1, \dots, s_p\}$,
\item $\shard(i,j,S) \eqdef \set{\b{x} \in \R^n}{x_i = x_j, \; x_i \ge x_k \text{ for all } k \in S, \; x_i \le x_k \text{ for all } k \in {]i,j[} \ssm S}$.
\end{itemize}
If~$|\{i,j\} \cap R| = 1$, then~$\ray(R)_i \ne \ray(R)_j$ so that~$\ray(R) \notin \shard(i,j,S)$.
Assume now that~$\{i,j\} \subseteq R$. Then we have~$\ray(R)_i = \ray(R)_k$ for any~$k \in {]i,j[} \cap R$ and~$\ray(R)_i < \ray(R)_k$ for any~$k \in {]i,j[} \ssm R$. Therefore, $\ray(R) \in \shard(i,j,S)$ if and only if~$S \subseteq {]i,j[} \cap R$.
Assume finally that~$\{i,j\} \subseteq [n] \ssm R$. Then we have~$\ray(R)_i = \ray(R)_k$ for any~$k \in {]i,j[} \cap R$ and~$\ray(R)_i > \ray(R)_k$ for any~$k \in {]i,j[} \cap R$. Therefore, $\ray(R) \in \shard(i,j,S)$ if and only if~${]i,j[} \cap R \subseteq S$.
\end{proof}

For example, the shard~$\shard(1,3,\varnothing)$ contains the rays~$\ray(\{1,3\})$, $\ray(\{1,2,3\})$, $\ray(\{1,3,4\})$ and~$\ray(\{4\})$. See \fref{fig:shards4}\,(middle) where~$\shard(1,3,\varnothing)$ is labeled by the arc~\raisebox{-.23cm}{\includegraphics[scale=.8]{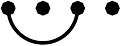}}\,.

\begin{lemma}
\label{lem:specialShard}
Let~$\sigma, \sigma'$ be two adjacent permutations, let~$\varnothing \ne R \subsetneq [n]$ (resp.~$\varnothing \ne R' \subsetneq [n]$) be such that~$\ray(R)$ (resp.~$\ray(R')$) is the ray of~$C(\sigma)$ not in~$C(\sigma')$ (resp.~of~$C(\sigma')$ not in~$C(\sigma)$), and let~$k, k'$ be such that~$R \ssm \{k\} = R' \ssm \{k'\}$. Assume without loss of generality that~$k < k'$. Then the common facet of~$C(\sigma)$ and~$C(\sigma')$ belongs to the shard~$\shard(k, k', R \cap R' \cap {]k,k'[})$.
\end{lemma}

\begin{proof}
As in Lemma~\ref{lem:linearDependence}, let~$i \in [n-1]$ be such that~${\sigma(i) = \sigma'(i+1) \defeq k}$ and~${\sigma(i+1) = \sigma'(i) \defeq k'}$ while~$\sigma(\ell) = \sigma'(\ell)$ for any~$\ell \notin \{i, i+1\}$. Define~$R \eqdef \sigma([i])$ and~$R' \eqdef \sigma'([i])$ and observe that ${R \ssm \{k\} = R' \ssm \{k'\}}$.
Then~$\ray(R)$ is the ray of~$C(\sigma) \ssm C(\sigma')$ while~$\ray(R')$ is the ray of~$C(\sigma') \ssm C(\sigma)$, and the rays of~$C(\sigma) \cap C(\sigma')$ are the rays~$\ray(\sigma([\ell])) = \ray(\sigma'([\ell]))$ for~$\ell \ne i$.
For any~$\ell > i$, we have~$\{k,k'\} \subseteq \sigma([\ell])$ and~$R \cap R' \subset \sigma([\ell])$ so that~$\ray(\sigma([\ell]))$ belongs to~$\shard(k, k', R \cap R' \cap {]k,k'[})$ by Lemma~\ref{lem:rayInShard}.
For any~$\ell < i$, we have~$\{k,k'\} \subseteq [n] \ssm \sigma([\ell])$ and~$\sigma([\ell]) \subseteq R \cap R'$ so that~$\ray(\sigma([\ell]))$ belongs to~$\shard(k, k', R \cap R' \cap {]k,k'[})$ by Lemma~\ref{lem:rayInShard}.
We conclude that all rays of the cone~$C(\sigma) \cap C(\sigma')$ are in the shard~$\shard(k, k', R \cap R' \cap {]k,k'[})$, and thus all the cone~$C(\sigma) \cap C(\sigma')$ is by convexity.
\end{proof}

For example, the common facet of the cones~$C(4132)$ and~$C(4312)$ of~$\Fan_4$ is supported by the shard~$\shard(1,3,\varnothing)$. See \fref{fig:shards4}\,(middle) where~$\shard(1,3,\varnothing)$ is labeled by the arc~\raisebox{-.23cm}{\includegraphics[scale=.8]{arc2}}\,.

\medskip
It turns out that the shards are precisely the pieces of the hyperplanes of~$\HA_n$ needed to delimit the cones of the quotient fan~$\Fan_\equiv$ for any lattice congruence~$\equiv$ of the weak order on~$\fS_n$. Conversely, to understand which sets of shards can be used to define a quotient fan, we need the forcing order between shards.
A shard~$\shard(i,j,S)$ is said to \defn{force} a shard~$\shard(k,\ell,T)$ if~$k \le i < j \le \ell$ and~$S = T \cap {]i,j[}$.
We denote this by~$\shard(i,j,S) \succ \shard(k,\ell,T)$.
For example, the forcing order on~$\shards_4$ is represented on \fref{fig:forcingOrder} together with one of its upper ideals.
All its upper ideals containing its three maximal elements are represented in \fref{fig:quotientLattice}.
The following statement uses shards to describe the lattice quotients of the weak order on~$\fS_n$.

\begin{figure}
	\capstart
	\centerline{\includegraphics[scale=.7]{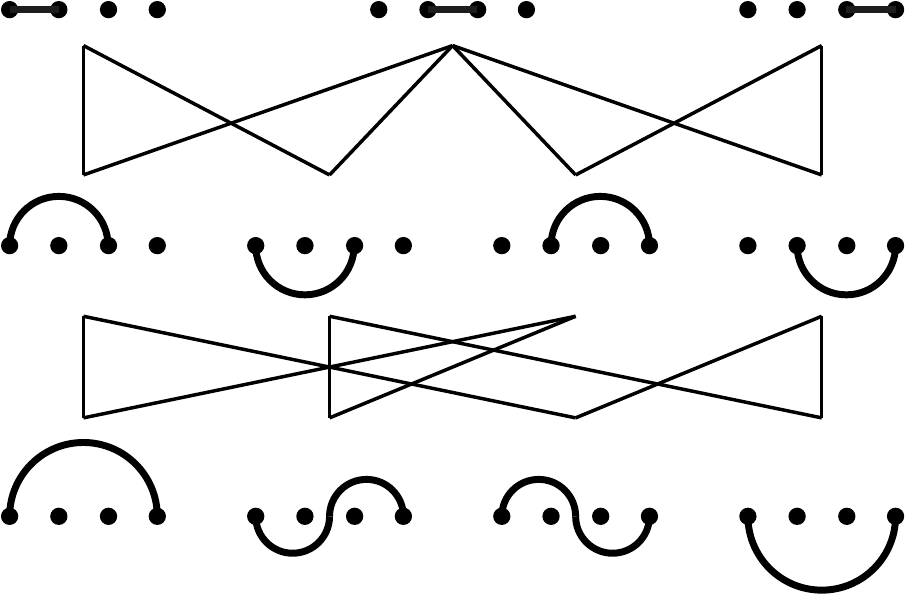} \qquad \includegraphics[scale=.7]{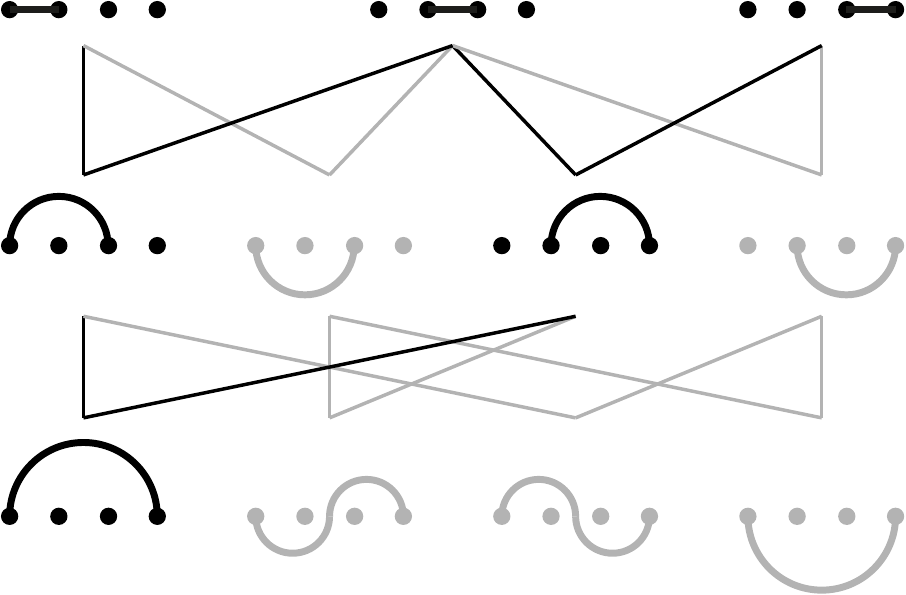}}
	\caption{The forcing order on~$\shards_4$ (left) and its upper ideal consisting of up arcs corresponding to the sylvester congruence~$\equiv^\textrm{sylv}$~(right).}
	\label{fig:forcingOrder}
\end{figure}

\begin{figure}[p]
	\capstart
	\centerline{\includegraphics[scale=.63]{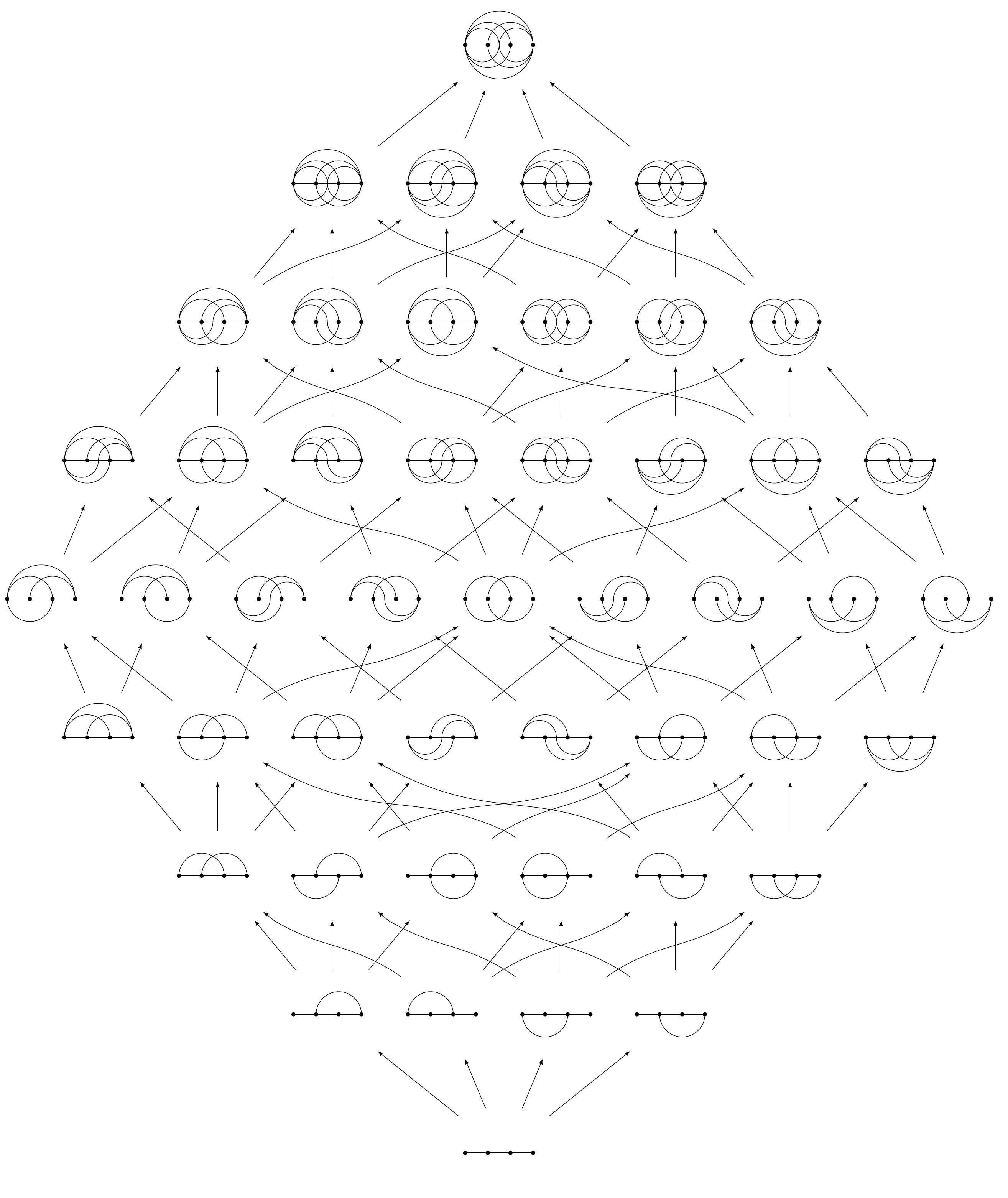}}
    \caption{The lattice of (essential) lattice congruences of the weak order on~$\fS_4$. Each lattice congruence~$\equiv$ is represented by its upper ideal of shards~$\shards_\equiv$, and each shard~$\shard(i,j,S)$ is represented by an arc with endpoints~$i$ and~$j$ and passing above the vertices of~$S$ and below those of~${]i,j[} \ssm S$. We only consider lattice congruences whose shards include all basic shards~$\shard(i,i+1,\varnothing)$, since otherwise their fan is not essential.}
    \label{fig:quotientLattice}
\end{figure}

\begin{theorem}[{\cite[Sect.~10.5]{Reading-FiniteCoxeterGroupsChapter}}]
\label{thm:shardIdeals}
For any lattice congruence~$\equiv$ of the weak order on~$\fS_n$, there is a subset~$\shards_\equiv$ of the shards of~$\shards_n$ such that the interiors of the maximal cones of the fan~$\Fan_\equiv$ are precisely the connected components of~$\hyp \ssm \bigcup \shards_\equiv$. Moreover, $\shards_\equiv$ is an upper ideal of the forcing order~$\prec$ and the map~${\equiv} \mapsto \shards_\equiv$ is a bijection between the lattice congruences of the weak order on~$\fS_n$ and the upper ideals of the forcing order~$\prec$.
\end{theorem}

For example, we have represented in Figures~\ref{fig:shards3}\,(right) and~\ref{fig:shards4}\,(right) the quotient fans~$\Fan_{\equiv^\textrm{sylv}}$ corresponding to the sylvester congruences~$\equiv^\textrm{sylv}$ on~$\fS_3$ and~$\fS_4$ respectively.
It is obtained 
\begin{itemize}
\item either by glueing together the chambers~$C(\sigma)$ of the permutations in the same sylvester class. These classes are given in \fref{fig:latticeQuotient4} for~$n=4$.
\item or by cutting the space with the shards of~$\shards_{\equiv^\textrm{sylv}}$. These shards are precisely the upper shards~$\shard(i, j, ]i,j[)$, whose arcs pass above all dots in between their endpoints. They form the ideal~$\shards_{\equiv^\textrm{sylv}}$ of~$\shards_4$ represented in \fref{fig:forcingOrder}\,(right).
\end{itemize}

\begin{remark}
We have already mentioned that we represent a shard~$\shard(i,j,S)$ by the arc with endpoints~$i$ and~$j$ and passing above the vertices of~$S$ and below those of~${]i,j[} \ssm S$.
Each region~$C$ of~$\Fan_\equiv$ then corresponds to a unique \defn{noncrossing arc diagram}~\cite{Reading-arcDiagrams}: a collection of arcs that pairwise do not intersect or share a common left endpoint or a common right endpoint.
Namely, the noncrossing arc diagram of a region~$C$ is given by the shards containing a down facet of~$C$.
This correspondence provides the canonical join representation of~$C$.
See~\cite{Reading-arcDiagrams} for precise definitions and details.
We also refer to N.~Reading's surveys~\cite{Reading-PosetRegionsChapter, Reading-FiniteCoxeterGroupsChapter} for further technology on the geometry of lattice quotients  (see also~Remark~\ref{rem:generalCase}).
\end{remark}

%%%%%%%%%%%%%%%%%%%%%%%%%%%%%%%%%%%%%%

\section{Quotientopes}

This section is devoted to the proof of Theorem~\ref{thm:quotientopes}.
We say that a function~$f : \shards_n \to \R_{>0}$ is \defn{forcing dominant} if
\[
f(\shard) \; > \sum_{\shard' \prec \shard} f(\shard')
\]
for any shard~$\shard \in \shards_n$. Such a function clearly exists, take for example~$f(\shard(i,j,S))$ to be~$n^{-(j-i)^2}$. For the remaining of the paper, we fix a forcing dominant function~$f$.

For a shard $\shard = \shard(i,j,S) \in \shards_n$ and a subset~$R \subseteq [n]$, we define the \defn{contribution}~$\gamma(\shard, R)$ of~$\shard$ to~$R$ to be~$1$ if~${|R \cap \{i,j\}| = 1}$ and~$S = R \cap {]i,j[}$, and $0$ otherwise.
It is crucial to observe that the definition of~$\gamma(\shard, R)$ only depends on~$R \cap [i,j]$.
Note also that~$\gamma(\shard, \varnothing) = \gamma(\shard, [n]) = 0$ for every shard~$\shard \in \shards_n$.
For a geometric interpretation of this definition, let~$\HA_n^{ij}$ denote the arrangement of the hyperplanes~$\Hyp{ik}$ and~$\Hyp{kj}$ for all~$k \in {]i,j[}$. Then $\shard(i,j,S)$ contributes to a subset~$\varnothing \ne R \subsetneq [n]$ if the ray~$\ray(R)$ lies in the (closed) region of~$\HA_n^{ij}$ containing~$\shard(i,j,S)$, but not on~$\shard(i,j,S)$.

We consider a lattice congruence~$\equiv$ of the weak order on~$\fS_n$. For a subset~$R \subsetneq [n]$, we define the \defn{height}~$h^f_\equiv(R) \in \R_{>0}$ to be
\[
h^f_\equiv(R) \; \eqdef \sum_{\shard \in \shards_\equiv} f(\shard) \, \gamma(\shard, R).
\]
Note that~$h^f_\equiv(\varnothing) = h^f_\equiv([n]) = 0$ by definition.
This height function fulfills the following property.

\begin{lemma}
\label{lem:inequality}
Let~$\sigma, \sigma'$ be two adjacent permutations. Let~$\varnothing \ne R \subsetneq [n]$ (resp.~$\varnothing \ne R' \subsetneq [n]$) be such that~$\ray(R)$ (resp.~$\ray(R')$) is the ray of~$C(\sigma)$ not in~$C(\sigma')$ (resp.~of~$C(\sigma')$ not in~$C(\sigma)$). Then
\[
h^f_\equiv(R) \, + \, h^f_\equiv(R') \; \ge \; h^f_\equiv(R \cap R') \, + \, h^f_\equiv(R \cup R')
\]
with equality if and only if the common facet of~$C(\sigma)$ and~$C(\sigma')$ belongs to a shard of~$\shards_\equiv$.
\end{lemma}

\begin{proof}
Let~$k, k'$ be such that~$R \ssm \{k\} = R' \ssm \{k'\}$.
Assume without loss of generality that~$k < k'$.
We consider a shard~$\shard = \shard(i,j,S) \in \shards_\equiv$ and evaluate its contributions to~$R$, $R'$, $R \cap R'$ and~$R \cup R'$.
Since~$\gamma(\shard, R)$ only depends on~$R \cap [i,j]$, observe that
\begin{itemize}
\item if~$\{k, k'\} \cap [i,j] = \varnothing$, then~$\gamma(\shard, R) = \gamma(\shard, R') = \gamma(\shard, R \cap R') = \gamma(\shard, R \cup R')$;
\item if~$\{k, k'\} \cap [i,j] = \{k\}$, then~$\gamma(\shard, R) = \gamma(\shard, R \cup R')$ and~$\gamma(\shard, R') = \gamma(\shard, R \cap R')$;
\item if~$\{k, k'\} \cap [i,j] = \{k'\}$, then~$\gamma(\shard, R) = \gamma(\shard, R \cap R')$ and~$\gamma(\shard, R') = \gamma(\shard, R \cup R')$.
\end{itemize}
Note also that $\gamma(\shard, R) = \gamma(\shard, R') = \gamma(\shard, R \cap R') = \gamma(\shard, R \cup R') = 0$ if~$S \ssm \{k,k'\} \ne R \cap R' \cap {]i,j[}$ by definition of the contributions.
We conclude that
\[
\gamma(\shard, R) \, + \, \gamma(\shard, R') \; = \; \gamma(\shard, R \cap R') \, + \, \gamma(\shard, R \cup R')
\]
for any shard~$\shard = \shard(i,j,S)$ for which~$\{k, k'\} \not\subseteq [i,j]$ or~$S \ssm \{k,k'\} \ne R \cap R' \cap {]i,j[}$.

Therefore, we are left with the contributions of the shards~$\shard(i,j,S)$ such that~$\{k, k'\} \subseteq [i,j]$ and~$S \ssm \{k,k'\} = R \cap R' \cap {]i,j[}$.
By definition of the forcing order, all these remaining shards are forced by the shard~$\shard_\bullet \eqdef \shard(k,k', R \cap R' \cap {]k,k'[})$.
Hence, we obtain that,
\[
h^f_\equiv(R) \, + \, h^f_\equiv(R') \, - \, h^f_\equiv(R \cap R') \, + \, h^f_\equiv(R \cup R') = \widetilde{h}^f_\equiv(R) \, + \, \widetilde{h}^f_\equiv(R') \, - \, \widetilde{h}^f_\equiv(R \cap R') \, + \, \widetilde{h}^f_\equiv(R \cup R'),
\]
where
\[
\widetilde{h}^f_\equiv(R) \eqdef \sum_{\substack{\shard \in \shards_\equiv \\ \shard \prec \shard_\bullet}} f(\shard) \, \gamma(\shard, R).
\]
%To conclude, observe that:
%\begin{itemize}
%\item $\gamma(\shard_\bullet, R) = \gamma(\shard_\bullet, R') = 1$ while~$\gamma(\shard_\bullet, R \cap R') = \gamma(\shard_\bullet, R \cup R') = 0$.
%\item By definition of the forcing order, all these remaining shards are forced by~$\shard_\bullet$.
%\item According to Lemma~\ref{lem:specialShard} and Theorem~\ref{thm:shardIdeals}, $\shard_\bullet$ is a shard of~$\shards_\equiv$ if and only if the cones~$C(\sigma)$ and~$C(\sigma')$ do not belong to the same cone of~$\Fan_\equiv$.
%\end{itemize}
Finally, according to Lemma~\ref{lem:specialShard} and Theorem~\ref{thm:shardIdeals}, $\shard_\bullet$ is a shard of~$\shards_\equiv$ if and only if the cones~$C(\sigma)$ and~$C(\sigma')$ belong to distinct cones of~$\Fan_\equiv$.
We therefore distinguish two cases.
\begin{enumerate}[(i)]
\item Assume first that~$C(\sigma)$ and~$C(\sigma')$ belong to the same cone of~$\Fan_\equiv$. Then~$\shard_\bullet$ is not in~$\shards_\equiv$. Since~$\shards_\equiv$ is an upper ideal of the forcing order, it implies that $\shards_\equiv$ contains no shard~$\shard$ with~$\shard \prec \shard_\bullet$. Therefore~$\widetilde{h}^f_\equiv(R) = \widetilde{h}^f_\equiv(R') = \widetilde{h}^f_\equiv(R \cap R') = \widetilde{h}^f_\equiv(R \cup R') = 0$ and we obtain
\[
h^f_\equiv(R) \, + \, h^f_\equiv(R') \; = \; h^f_\equiv(R \cap R') \, + \, h^f_\equiv(R \cup R').
\]

\item Assume now that~$C(\sigma)$ and~$C(\sigma')$ belong to distinct cones of~$\Fan_\equiv$. Then~$\shard_\bullet$ is in~$\shards_\equiv$. Since~$\gamma(\shard_\bullet, R) = \gamma(\shard_\bullet, R') = 1$, this implies that $\widetilde{h}^f_\equiv(R) \ge  f(\shard_\bullet)$ and $\widetilde{h}^f_\equiv(R') \ge  f(\shard_\bullet)$. Moreover, since~$\gamma(\shard_\bullet, R \cap R') = \gamma(\shard_\bullet, R \cup R') = 0$, we have~$\widetilde{h}^f_\equiv(R \cap R') \le \sum_{\shard \prec \shard_\bullet} f(\shard')$ and~$\widetilde{h}^f_\equiv(R \cup R') \le \sum_{\shard \prec \shard_\bullet} f(\shard')$. We conclude that
\[
h^f_\equiv(R) \, + \, h^f_\equiv(R') \, - \, h^f_\equiv(R \cap R') \, - \, h^f_\equiv(R \cup R') \; \ge \; 2 \, f(\shard_\bullet) \, - \, 2 \sum_{\shard \prec \shard_\bullet} f(\shard) \; > \; 0.
\]
since~$f$ is forcing dominant. This concludes the proof. \qedhere
\end{enumerate}
\end{proof}

We finally obtain the proof of Theorem~\ref{thm:quotientopes}.

\begin{corollary}
\label{coro:quotientopes}
For any lattice congruence~$\equiv$ of the weak order on~$\fS_n$, and any forcing dominant function~$f : \shards_n \to \R_{>0}$, the quotient fan~$\Fan_\equiv$ is the normal fan of the polytope
\[
P_\equiv^f \eqdef \bigset{\b{x} \in \R^n}{\dotprod{\ray(R)}{\b{x}} \le h^f_\equiv(R) \text{ for all } \varnothing \ne R \subsetneq [n]}.
\]
In particular, the graph of~$P_\equiv^f$ oriented in the linear direction~$\alpha \eqdef (-n+1, -n+3, \dots, n-3, n-1)$ is the Hasse diagram of the quotient of the weak order by~${\equiv}$.
\end{corollary}

\begin{proof}
We just combine the polytopality criterion of Proposition~\ref{prop:polytopalSubfanFan} with the statements of Lemmas~\ref{lem:linearDependence} and~\ref{lem:inequality}, to obtain the polytopality of the quotient fan~$\Fan_\equiv$.
The end of the statement then follows from~Theorem~\ref{thm:fanQuotient}.
\end{proof}

Note that the inequality description of Corollary~\ref{coro:quotientopes} is in general redundant. More precisely, the inequalities corresponding to the rays of~$\Fan_n$ that are not rays of~$\Fan_\equiv$ are irrelevant.

We call \defn{quotientope} the resulting polytope~$P_\equiv^f$. See Figures~\ref{fig:permAssoCube}, \ref{fig:relevantQuotientopes} and~\ref{fig:quotientopeLattice} for illustrations. Note that not all quotientopes are simple since not all quotient fans are simplicial.

\begin{figure}
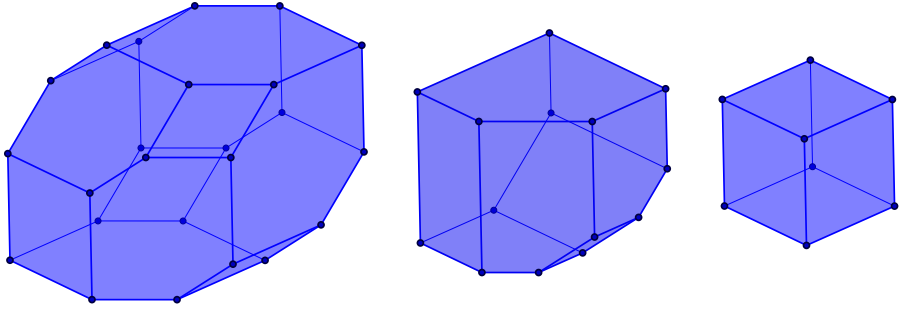

	\capstart
	\centerline{\scalebox{.8}{$\vcenter{\input{figures/permutahedron}} \hspace{-9cm} \vcenter{\input{figures/associahedron}} \hspace{-10.5cm} \vcenter{\input{figures/cube}} \hspace{-1.3cm}$}}
	\caption{Permutahedron (left), associahedron (middle) and cube (right) as quotientopes.}
	\label{fig:permAssoCube}
\end{figure}

\begin{figure}[p]
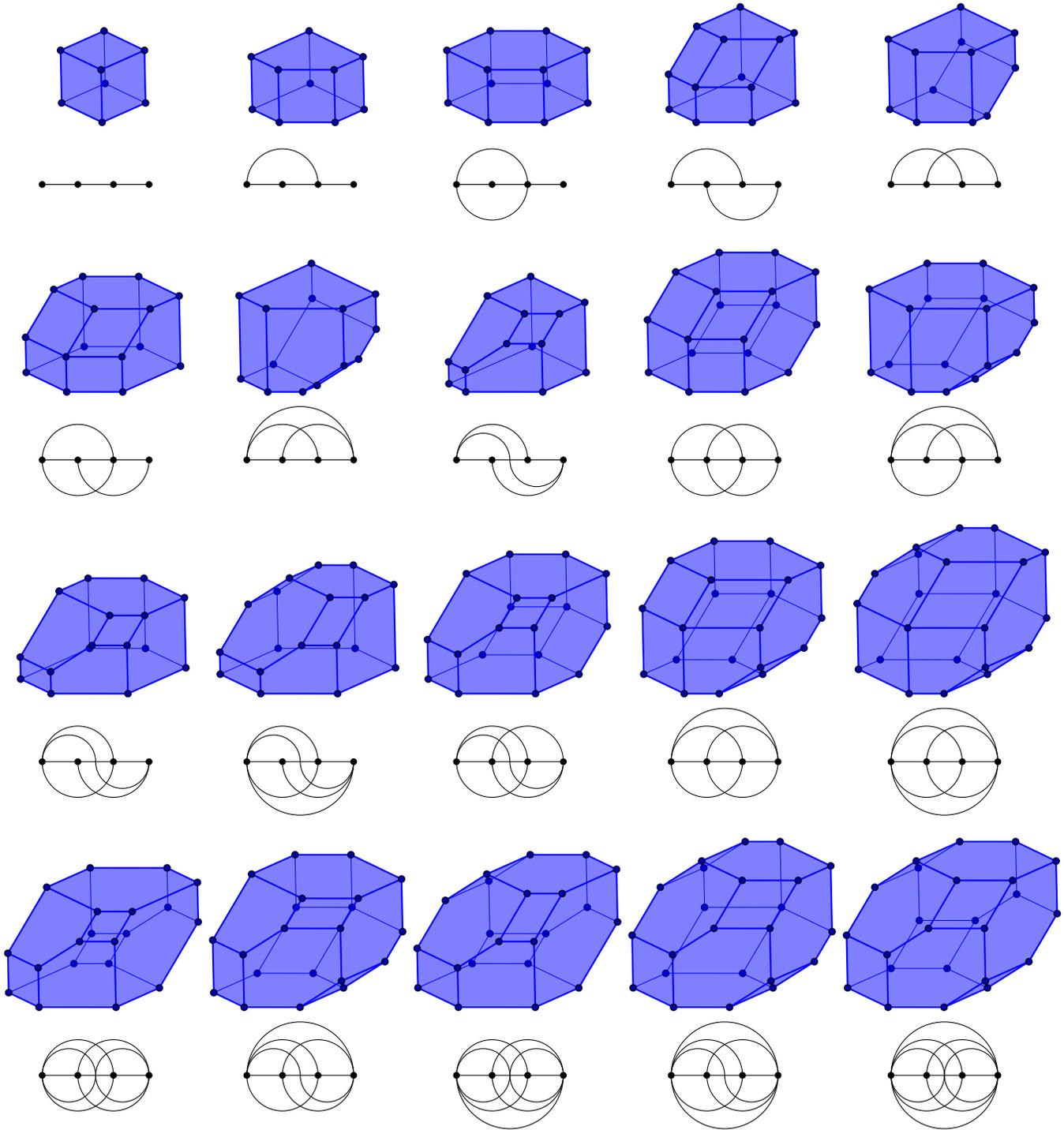

	\capstart
    \centerline{
    \begin{tabular}{c@{}c@{}c@{}c@{}c}
    \input{figures/quotientope0} & \input{figures/quotientope1} & \input{figures/quotientope2} & \input{figures/quotientope3} & \input{figures/quotientope4} \\[-.2cm]
    \input{figures/quotient0} & \input{figures/quotient1} & \input{figures/quotient2} & \input{figures/quotient3} & \input{figures/quotient4} \\[-.1cm]
    \input{figures/quotientope5} & \input{figures/quotientope6} & \input{figures/quotientope7} & \input{figures/quotientope8} & \input{figures/quotientope9} \\[-.1cm]
    \input{figures/quotient5} & \input{figures/quotient6} & \input{figures/quotient7} & \input{figures/quotient8} & \input{figures/quotient9} \\[-.1cm]
    \input{figures/quotientope10} & \input{figures/quotientope11} & \input{figures/quotientope12} & \input{figures/quotientope13} & \input{figures/quotientope14} \\[-.1cm]
    \input{figures/quotient10} & \input{figures/quotient11} & \input{figures/quotient12} & \input{figures/quotient13} & \input{figures/quotient14} \\[.1cm]
    \input{figures/quotientope15} & \input{figures/quotientope16} & \input{figures/quotientope17} & \input{figures/quotientope18} & \input{figures/quotientope19} \\[-.1cm]
    \input{figures/quotient15} & \input{figures/quotient16} & \input{figures/quotient17} & \input{figures/quotient18} & \input{figures/quotient19} \\[.1cm]
    \end{tabular}
    }
    \caption{All $3$-dimensional quotientopes up to symmetries. There are $47$ (essential) lattice congruences on~$\fS_4$ represented in \fref{fig:quotientLattice}, but only $20$ up to horizontal and vertical symmetry. Each lattice congruence~$\equiv$ is represented by its upper ideal of shards~$\shards_\equiv$, and each shard~$\shard(i,j,S)$ is represented by an arc with endpoints~$i$ and~$j$ and passing above the vertices of~$S$ and below those of~${]i,j[} \ssm S$.}
% number of lattice congruences: 1, 2, 7, 60, 3444, 11402948, 129187106461769 -- http://oeis.org/A091687
% number of connected lattice congruences: 1 1 1 4 47 3322 11396000
    \label{fig:relevantQuotientopes}
\end{figure}

\begin{figure}[p]
	\capstart
	\centerline{\includegraphics[scale=.38]{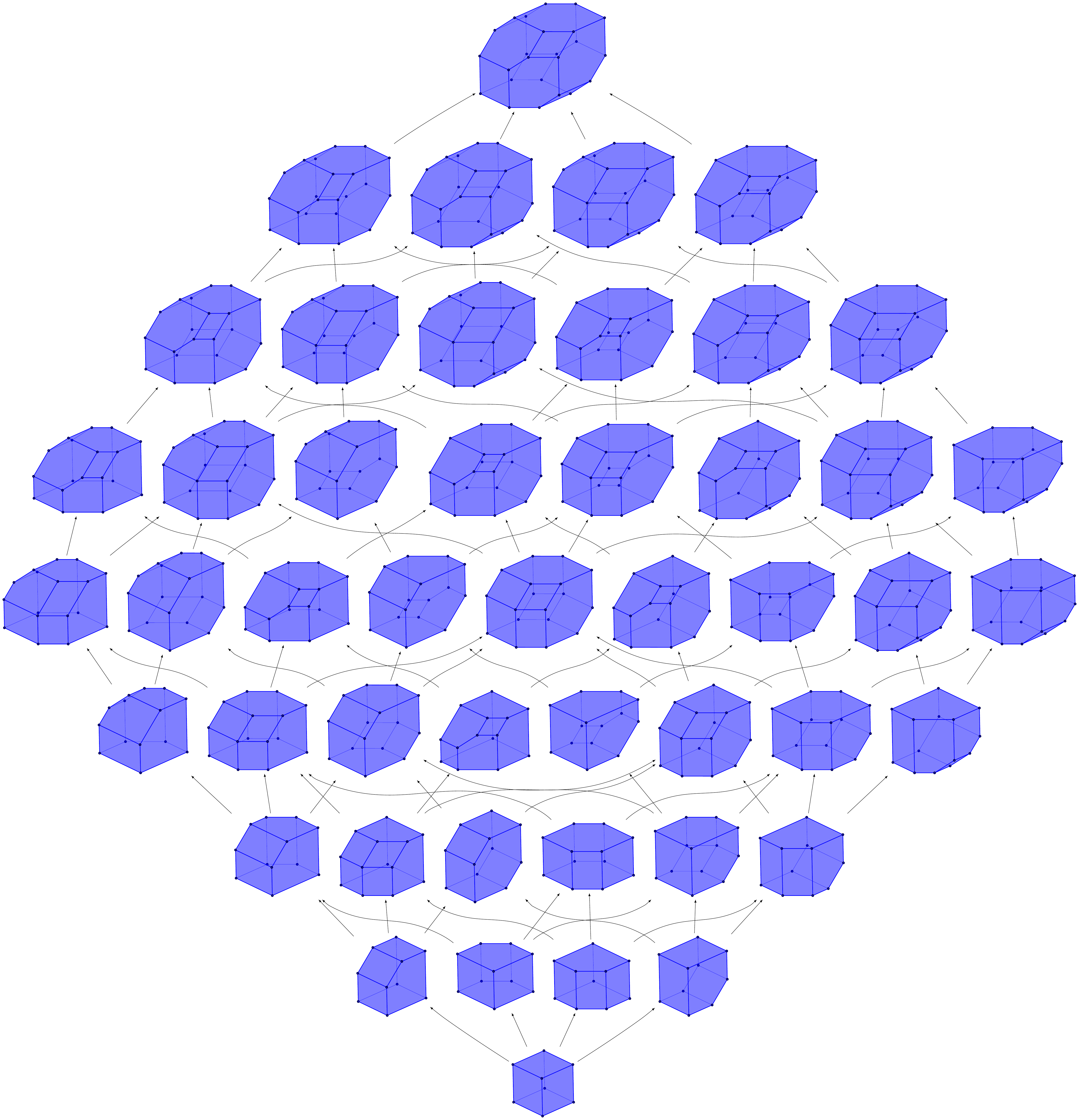}}
    \caption{The quotientope lattice for~$n = 4$: all quotientopes ordered by inclusion (which corresponds to refinement of the lattice congruences). We only consider lattice congruences whose shards include all basic shards~$\shard(i,i+1,\varnothing)$, since otherwise their fan is not essential.}
    \label{fig:quotientopeLattice}
\end{figure}

\begin{remark}[Forcing dominance]
Note that the forcing dominance condition could even be weakened to depend on the lattice congruence~$\equiv$. More precisely, the construction and the proof of Lemma~\ref{lem:inequality} and thus of Corollary~\ref{coro:quotientopes} still work for any function~$f : \shards_n \to \R_{>0}$ such that for any shard~$\shard \in \shards_\equiv$,
\[
f(\shard) \; > \sum_{\substack{\shard' \in \shards_\equiv \\ \shard' \prec \shard}} f(\shard').
\]
\end{remark}

\begin{remark}[\defn{Insidahedra}, \defn{outsidahedra} and \defn{removahedra}]
By definition, the quotientopes are generalized permutahedra~\cite{Postnikov, PostnikovReinerWilliams} as their normal fans coarsen the braid fan. This means in particular that they are obtained by gliding inequalities of the permutahedron orthogonally to their normal vectors. Note that in our construction, the inequalities are glided \defn{inside} the permutahedron. More precisely, if~$\Fan_\equiv$ refines~$\Fan_{\equiv'}$, then $P_\equiv^f$ contains~$P_{\equiv'}^f$. For example, the cube (quotientope of the coarsest congruence so that~$\Fan_\equiv$ is essential) is contained in all quotientopes such that~$\Fan_\equiv$ is essential, while the permutahedron (quotientope of the finest congruence) contains all quotientopes. See~\fref{fig:quotientopeLattice} for illustration. This construction thus contrasts with the classical construction of the associahedron~\cite{Loday} and its generalizations~\cite{HohlwegLange, LangePilaud, Pilaud-signedTreeAssociahedra, PilaudPons-permutrees}, which are all obtained by gliding inequalities \defn{outside} the permutahedron. More precisely, the classical associahedron is obtained by \defn{removing} certain inequalities from the facet description of the classical permutahedron. Note that the similar construction does not work in general: for example, the fan~$\Fan_\equiv$ of the top right congruence of \fref{fig:relevantQuotientopes} is not realized by the intersection of the half-spaces defining facets of the classical permutahedron normal to the rays of~$\Fan_\equiv$.
\end{remark}

\begin{remark}[Towards quotientopes for arbitrary hyperplane arrangements?]
\label{rem:generalCase}
As already mentioned, Theorem~\ref{thm:fanQuotient} actually holds in much more generality (see~\cite{Reading-PosetRegionsChapter} for a detailed survey). Consider a central hyperplane arrangement~$\HA$ defining a fan~$\Fan$, and let~$B$ be a distinguished chamber of~$\Fan$. For any chamber~$C$ of~$\Fan$, define its inversion set to be the set of hyperplanes of~$\HA$ that separate~$B$ from~$C$. The \defn{poset of regions}~$\Pos(\HA, B)$ is the poset whose elements are the chambers of~$\Fan$ ordered by inclusion of inversion sets. A.~Bj\"orner, P.~Edelman and G.~Ziegler discuss in~\cite{BjornerEdelmanZiegler} some conditions for this poset of regions to be a lattice: $\Pos(\HA,B)$ is always a lattice when the fan~$\Fan$ is simplicial, and the chamber~$B$ must be a simplicial for~$\Pos(\HA,B)$ to be a lattice. In~\cite{Reading-HopfAlgebras}, N.~Reading proves that when~$\Pos(\HA,B)$ is a lattice, any lattice congruence~$\equiv$ of~$\Pos(\HA,B)$ defines a complete fan~$\Fan_\equiv$ obtained by glueing together the cones of the fan~$\Fan$ that belong to the same congruence class of~${\equiv}$. The polytopality of this quotient fan~$\Fan_\equiv$ however remains open in general. Although the polytopality criterion of Proposition~\ref{prop:polytopalSubfanFan} seems a promising tool to tackle this problem when~$\Fan$ is simplicial, the general case seems much more intricate. Let us observe that we benefited from three specific features of the Coxeter arrangement of type~$A$:
\begin{itemize}
\item we used the simpliciality of the arrangement,
\item we used the action of~$\fS_n$ to transport our understanding of the linear dependencies from the initial chamber to any other chamber,
\item these linear dependencies are very simple in type~$A$ (only $3$ or $4$ terms and $0/1$ coefficients).
\end{itemize}
These properties hold for any finite Coxeter group (for the last property though, the linear dependencies can get up to $5$ terms, and some coefficients equal to~$2$ appear in non-simply-laced types). This suggests that the strategy of this paper could produce polytopal realizations when the hyperplane arrangement is the Coxeter arrangement of a finite Coxeter group.
\end{remark}

%%%%%%%%%%%%%%%%%%%%%%%%%%%%%%%%%%%%%%

\addtocontents{toc}{\vspace{.3cm}}
\section*{Acknowledgements}

We thank N.~Reading for comments on a preliminary version of this paper and an anonymous referee for many relevant suggestions that greatly improved the presentation.

%%%%%%%%%%%%%%%%%%%%%%%%%%%%%%%%%%%%%%

\bibliographystyle{alpha}
\bibliography{quotientopes}
\label{sec:biblio}

\end{document}